\documentclass[reqno,11pt]{amsart}

\usepackage{graphicx} 
\usepackage{todonotes}
\usepackage{aliascnt} 
\usepackage[english]{babel}
\usepackage{hyperref}
\usepackage{amssymb,amsmath,amsthm,color,mathtools,tikz,subfig,dirtytalk,csquotes,soul,orcidlink,enumitem}
\usepackage[dvipsnames]{xcolor}
\usepackage[capitalise, nameinlink]{cleveref}
\usepackage{comment}
\usepackage{amsmath}
\usepackage{amsthm}
\usepackage[left=2.8cm,right=2.8cm,top=2.8cm,bottom=2.8cm]{geometry}

\usepackage[
    style=numeric,    
    giveninits=true,     
    isbn=false,          
    url=false,
    doi=false,            
    eprint=true,         
    backend=biber,        
    sorting=nty,   
    sortcites=true, 
    maxbibnames=99
]{biblatex}
\numberwithin{equation}{section}

\theoremstyle{definition} 

\crefname{example}{Example}{Examples}

\theoremstyle{plain}

\newaliascnt{theo}{example}       
\newtheorem{theo}[theo]{Theorem}  
\aliascntresetthe{theo}           
\crefname{theo}{Theorem}{Theorems} 

\newaliascnt{prop}{example}

\aliascntresetthe{prop}
\crefname{prop}{Proposition}{Propositions}

\newaliascnt{lemma}{example}
\newtheorem{lemma}[lemma]{Lemma}
\aliascntresetthe{lemma}
\crefname{lemma}{Lemma}{lemmata}

\newaliascnt{cor}{example}

\aliascntresetthe{cor}
\crefname{cor}{Corollary}{Corollaries}

\theoremstyle{definition}

\newaliascnt{dhef}{example}

\aliascntresetthe{dhef}
\crefname{dhef}{Definition}{Definitions}

\newaliascnt{hyp}{example}

\aliascntresetthe{hyp}
\crefname{hyp}{Assumption}{Assumptions}

\theoremstyle{remark}

\newaliascnt{remark}{example}
\newtheorem{remark}[remark]{Remark}
\aliascntresetthe{remark}
\crefname{remark}{Remark}{Remarks}

\newaliascnt{obs}{example}
\newtheorem{obs}[obs]{Observation}
\aliascntresetthe{obs}
\crefname{obs}{Observation}{Observations}

\def\elle#1{L^{#1}}
\def\elleom#1{L^{#1}(\Omega)}

\def\Sd{{\mathbb{S}^{d-1}}}
\renewcommand{\phi}{\varphi}

\def\C{\mathcal{C}}
\def\sub{\subseteq}

\newlist{assumption}{enumerate}{1}
\setlist[assumption,1]{label=(\roman*),ref=(\roman*)}
\crefname{assumptioni}{assumption}{assumptions}

\def\Haus{{\mathcal{H}^{d-1}}}

\def\O{\Omega}

\def \e{\varepsilon}

\def\L{\mathcal{L}}

\def\N{\mathbb{N}}
\def\Z{\mathbb{Z}}

\def\R{\mathbb{R}}
\def\Rd{\mathbb{R}^d}

\DeclareRobustCommand{\rchi}{{\mathpalette\irchi\relax}}
\newcommand{\irchi}[2]{\raisebox{\depth}{$#1\chi$}} 

\title{A Homogenized Maz'ya--Shaposhnikova formula}

\author{Andrea Braides
\orcidlink{0000-0003-1612-2785}}\email{braides@mat.uniroma2.it}
\address{Department of Mathematics, University of Rome Tor Vergata, via della Ricerca Scientifica~1, 00133 Rome, Italy}
\author{Marco Picerni
\orcidlink{0009-0004-4364-4831}}\email{mpicerni@sissa.it}
\address{SISSA, via Bonomea 265, 34136, Trieste, Italy}
\author{Alec Jacopo Almo Schiavoni Piazza
\orcidlink{0000-0003-2652-7081}}\email{aschiavo@sissa.it}
\address{SISSA, via Bonomea 265, 34136, Trieste, Italy}



\begin{document}

\begin{abstract}
We study the asymptotic behaviour of fractional Gagliardo-type energies
with a periodic, bounded and bounded away from zero coefficient depending on both
variables, in the joint limit of vanishing fractional exponent
$s\to0^+$ and vanishing period $\varepsilon\to0^+$. Assuming without loss of generality that
$\varepsilon^{-2s}$ converges to some $\lambda\in[1,+\infty]$, we prove
that the energies, rescaled by $s$, $\Gamma$-converge in the weak
topology of $L^2(\mathbb{R}^d)$ to a multiple of the squared $L^2$
norm, whose constant is given by an effective coefficient $a_\lambda$
defined through a cell problem. The coefficient $a_\lambda$ results from
the competition between the averaging of long-range interactions, which
drives the Maz'ya--Shaposhnikova limit, and the nonlocal energy of
microscopic oscillations at finite range. It interpolates between the
harmonic mean ($\lambda=1$) and the arithmetic mean ($\lambda=+\infty$)
of a coefficient obtained by averaging interactions ``at infinity''. 
The proof relies on Riemann--Lebesgue-type lemmas for singular kernels, 
on a blow-up argument for the lower bound, and on the logarithmic-scaling 
properties of the energies. The case $1<p<\infty$ is also discussed.
\end{abstract}

\maketitle

\tableofcontents

\section{Introduction}\label{sec:introduction}
The Maz'ya--Shaposhnikova formula  describes the asymptotic behaviour of the Gagliardo seminorm as the fractional exponent $s$ tends to zero. More precisely, if for $s\in(0,1)$ the Gagliardo seminorm of $u\colon\Rd\to\R$ is given by 
\[
[u]_{H^s(\Rd)}^2\coloneqq\int_{\Rd}\int_{\Rd}\frac{|u(x)-u(y)|^2}{|x-y|^{d+2s}}\,dx\,dy ,
\]
then, for every $u\in H^{s_0}(\Rd)$ for some $s_0\in(0,1)$,
\begin{equation}\label{intro:MS}
\lim_{s\to0^+}s\,[u]_{H^s(\Rd)}^2=\sigma_{d-1}\|u\|^2_{L^2(\Rd)},
\end{equation}
where $\sigma_{d-1}$ is the $(d-1)$-dimensional Hausdorff measure of the unit sphere $\Sd$  \cite{MazyaShaposhnikova2002}. Hence, although the seminorm is defined through nonlocal interactions between pairs of points, its asymptotic behaviour as $s\to0^+$ is a zero-order functional.

In recent years, the Maz'ya--Shaposhnikova formula has been extended to several settings beyond its original Euclidean formulation. In metric measure spaces, volume growth at infinity (and, in particular, the asymptotic volume ratio) plays a central role in the construction of suitable convolution kernels \cite{Han-MS-BBM,HanPinamontirigidity}. Such conditions are guaranteed, for example, by generalized Bishop--Gromov inequalities \cite{HanPinamontiXu2025}. Moreover, necessary and sufficient conditions for the validity of a Maz'ya--Shaposhnikova limit for general convolution kernels are established in \cite{DaDfGioPi-Necess-Suff}: such conditions amount to a control on the escape of kernel mass to infinity and the attenuation of short-range interactions.
Related asymptotics have been obtained for fractional perimeters \cite{Quartet,DaFaPi1}, for fractional energies associated with the Ornstein--Uhlenbeck semigroup \cite{GaussianFracPer}, and for heat-kernel-based fractional perimeters on Riemannian manifolds \cite{CaselliGennaioli}.
 
The $\Gamma$-convergence counterpart of \eqref{intro:MS} has been established in two regimes. In \cite{CriDLKNP23}, functions are extended by zero outside a bounded domain, so that their value at infinity is zero. In \cite{DaFaPi1}, localized Gagliardo seminorms, such as those defining fractional perimeters, are considered for functions that need not be globally square-integrable: the limit functional describes the interaction between the values of $u$ in $\Omega$ and its values at infinity, thus connecting \eqref{intro:MS} with the asymptotics of fractional perimeters. In the Euclidean setting, the same work establishes $\Gamma$-convergence in the weak-$L^2(\Omega)$ topology for prescribed exterior data that need not vanish at infinity. In both cases the behaviour of $u$ at infinity is fixed a priori, and the $\Gamma$-limit only acts on its values in a bounded region.

\smallskip
In this paper, we investigate how \eqref{intro:MS} is affected by periodic heterogeneities in the interaction kernel. We consider fractional energies whose interactions are weighted by a periodic coefficient depending on the spatial variables and bounded from above and below by positive constants. The resulting problem combines two limiting processes, the vanishing of the fractional exponent and the homogenization of the oscillating coefficient. Interactions at distances larger than the macroscopic scale see the behaviour of $u$ at infinity, as in \eqref{intro:MS}. Interactions at distances between the period $\e$ and the macroscopic scale see oscillations as periodic perturbations with zero mean of the frozen macroscopic profile of $u$, with an effect reminiscent of \cite{PiniSantilli2026}, where a Maz'ya--Shaposhnikova formula is given for interactions in the torus. Since, as $s\to0^+$, the kernel distributes its mass uniformly over logarithmic scales, the relative weight of these two ranges is governed by $\e^{-2s}=e^{2s|\log\e|}$.
 
Our main result characterizes this interplay through a family of effective coefficients $a_\lambda$, defined by a cell-minimization problem, and establishes the corresponding $\Gamma$-convergence of the energies in the weak topology of $L^2(\Rd)$. The parameter $\lambda\coloneqq\lim_{\e\to0}\e^{-2s}\in[1,+\infty]$ measures the balance between the two ranges of interactions: for $\lambda=1$ only the behaviour at infinity contributes, whereas as $\lambda$ increases the periodic interactions penalize more and more the oscillations of the corrector.
The case $\lambda=+\infty$ corresponds to a {\em separation of scales}: the limit is obtained by first letting $\varepsilon\to0$ and then taking the Maz'ya–Shaposhnikova limit $s\to0$.
The precise statement is the following.

\begin{theo}[A homogenized Maz'ya--Shaposhnikova theorem]\label{thm:main}
Let $a\colon\Rd\times\Rd\to\R$ be a  Borel function such that there exist $0<\alpha\leq\beta$ for which $\alpha\leq a(x,y)\leq\beta$ for all $x,y\in\Rd$, and $1$-periodic; that is, $a(x+k,y+j)= a(x,y)$ for all $k,j\in\mathbb Z^d$. Let $\e>0$ and $s\coloneqq s(\e)\in(0,1)$. Let
\begin{equation}\label{intro:fe}
F_\e(u)\coloneqq s\int_{\Rd}\int_{\Rd}a\Big(\frac{x}{\e},\frac{y}{\e}\Big)\frac{|u(x)-u(y)|^2}{|x-y|^{d+2s}}\,dx\,dy
\end{equation}
be defined in $L^2(\Rd)$. If $\lim_{\e\to0}s=0$ and
\begin{equation*}
\lim_{\e\to0}\e^{-2s}=\lambda\in[1,+\infty],
\end{equation*}
then
\begin{equation*}
\Gamma\text{-}\lim_{\e\to0}F_\e(u)=\frac{\sigma_{d-1}}{2}\,a_\lambda\|u\|^2_{L^2(\Rd)},
\end{equation*}
for all $u\in L^2(\Rd)$, where the $\Gamma$-limit is computed with respect to the sequential weak topology of $L^2(\Rd)$, $a_\lambda$ is given by 
\begin{eqnarray}\label{eq:a-lambda-0}\nonumber
&&\hskip-1.3cm a_\lambda\coloneqq\inf\bigg\{\int_{Q}\int_{Q}(a(x,y)+a(y,x))|1+\phi(x)|^2\,dx\,dy\\&&\ \quad+(\lambda-1)\int_{Q}\int_{Q} a(x,y)|\phi(x)-\phi(y)|^2\,dx\,dy:\phi\in L^2(Q)\text{ 
and }\int_Q\phi(x)\,dx=0\bigg\}
\end{eqnarray}
if $\lambda<+\infty$, while
\begin{eqnarray}\label{eq:a-lambda-infty}
a_{+\infty}\coloneqq\ 2\int_{Q}\int_{Q}a(x,y)dx\,dy,
\end{eqnarray}
where $Q=(0,1)^d$ is the unit cube in $\Rd$, and $\sigma_{d-1}$ is the $(d-1)$-Hausdorff measure of the unit sphere $\Sd$ in $\Rd$.
\end{theo}

The $\Gamma$-convergence counterpart of the Maz'ya--Shaposhnikova formula is recovered as a particular case of our theorem, when we take the coefficient $a$ equal to the constant $1$, in which case the minimizer for the problem in \eqref{eq:a-lambda-0} is $\varphi=0$ independently of $\lambda$, and $a_\lambda=2$. 

Theorem \ref{thm:main} can be proved for Gagliardo seminorms with an arbitrary $p>1$. We restrict 
ourselves to the case $p=2$ so as to focus on the novelties of this work and not to carry another parameter through the proofs. The asymptotic results hold unchanged as in the $p$ version of the Maz'ya--Shaposhnikova formula,
upon defining $\lambda$ as the limit of $\e^{-sp}$ and  using the power $p$ in the definition of $a_\lambda$ (see Theorem \ref{thm:main2}).

For $a$ non-constant, the form of the coefficient $a_\lambda$ highlights a new nonlocal homogenization phenomenon in competition with the effect of the interactions ``with points at infinity'' that drive the Maz'ya--Shaposhnikova limit. This can be explained by comparing Theorem \ref{thm:main} with the Maz'ya--Shaposhnikova limit for a (smooth) function $u$, whose support is contained in a ball $B_{R_0}$. In the latter case, the limit of
$s[u]_{H^s(\Rd)}^2$ as $s\to 0$, using the symmetry between integration in $x$ and $y$ and noting that all interactions in $B_R$ can be neglected, is the same as that of 
$$
\lim_{s\to 0}2s\int_{B_{R_0}}\int_{\Rd\setminus B_{R}}\frac{|u(x)-u(y)|^2}{|x-y|^{d+2s}}\,dx\,dy= \bigg(\lim_{s\to 0} 2s\int_{\Rd\setminus B_{R}}\frac{1}{|\xi|^{d+2s}}\,d\xi\bigg)\int_{B_{R_0}}|u(x)|^2\,dx
$$
for any $R>R_0$. Since the limit is independent of $R$, a direct computation gives the result.

In the presence of an oscillating coefficient, the $\Gamma$-limit can be different from the pointwise limit of $F_\e(u)$ at $u$ fixed. Examining a sequence $u_\e$ weakly converging to $u$, we have two issues:

1) the term given by ``interactions going to infinity'' is now of the form
$$
s\int_{B_{R_0}}\int_{\Rd\setminus B_{R}}
\Big(a\Big(\frac{x}{\e},\frac{y}{\e}\Big)
+ a\Big(\frac{y}{\e},\frac{x}{\e}\Big)\Big)
\frac{|u_\e(x)|^2}{|x-y|^{d+2s}}\,dx\,dy.
$$
By first integrating in $y$, we see that this double integral is asymptotically equivalent as $s\to 0$ to an oscillating integral with a $1$-periodic coefficient. A Riemann--Lebesgue type lemma allows one to see that such an integral is indeed of the form
\begin{equation}\label{infinity-hom}
\frac{\sigma_{d-1}}{2} \int_{B_{R_0}}\bigg(\int_Q \Big(a\Big(\frac{x}{\e},\xi\Big)
+ a\Big(\xi,\frac{x}{\e}\Big)\Big)d\xi\bigg) |u_\e|^2dx.
\end{equation}
A clarifying example is when $a(x,y)= a_0(x)$ (or its symmetric equivalent version $\frac12(a_0(x)+a_0(y))$, for which this integral is 
$$
\frac{\sigma_{d-1}}{2} \int_{B_{R_0}}\Big(a_0\Big(\frac{x}{\e}\Big)+\overline a_0\Big) |u_\e|^2dx,
$$
with $\overline a_0=\int_Q a_0(\xi)d\xi$ representing the contribution ``averaged at infinity'';

2)  due to the presence of oscillations, the interactions at finite distance are not negligible any more. We have an additional term giving a non-vanishing contribution in the limit, which is of the form 
\begin{equation}\label{penalization_term}
s\iint_{\{\e\le |x-y|\le 1\}}a\Big(\frac{x}{\e},\frac{y}{\e}\Big)\frac{|u_\e(x)-u_\e(y)|^2}{|x-y|^{d+2s}}\,dx\,dy.
\end{equation}
The conditions on $s$ and $\e$ defining $\lambda$, and a use of a second Riemann--Lebesgue type lemma give the double-integral contribution to the definition of $a_\lambda$.


\smallskip
It is interesting to compare Theorem \ref {thm:main} with the limit taken as $s\to 1^-$ in the spirit of the Bourgain--Brezis--Mironescu formula \cite{bourgain:hal-00747692}. The asymptotic behaviour of the corresponding $F_\e$, with the scaling factor $s$ replaced by $1-s$ in \eqref{intro:fe}, has been described in \cite{BMS26} (see also \cite{BBD}), showing a behaviour dual to ours: the relevant oscillations in that case are those at scale smaller than $\e$, which are negligible in Theorem \ref {thm:main}. These small-scale oscillations lead to a quadratic functional in $\nabla u$, obtained by homogenizing a quadratic functional with coefficient $a(x,x)$, concentrated on the diagonal; that is, due to interactions at  distance less than $\e$, as opposed to the coefficient $\overline{a}^{\rm c}$ (see Definition \ref{def:oac}) that derives from averaging interactions at infinity. An additional contribution in the BBM limit is given by interactions at scales between $\e$ and $1$ (as in item (2) above). In contrast to the case $s\to 0$, in these interactions only the average of the oscillations matters, and they do not influence the homogenization on a lower scale. The overall limit is a convex combination of the two contributions, with the small-scale contribution now weighted by the limit of $\e^{2(1-s)}$.

We remark that in this paper we do not prove a compactness statement, which amounts to a uniform bound in $\elle2$. This, on a bounded domain $\O$, is proved through a fractional Hardy inequality \cite{CriDLKNP23,DaFaPi1}, but it fails on the whole $\Rd$. On the other hand, if we restrict our attention to functions belonging to $\elleom2$ and extended to zero outside of $\O$, such a compactness statement is available because $a(x,y)\geq\alpha>0$. The proof of the theorem remains unchanged in this constrained case, so that it may be applied to the convergence of minimum problems as in \cite{CriDLKNP23}.

\smallskip
\noindent{\bf Structure of the paper and core technical ideas.}
The whole paper is devoted to the proof of Theorem \ref{thm:main}. 
The key observation is that, due to the oscillations forced by the periodic coefficient, in the limit $s\to 0$ we may only have weak convergence in $L^2$. This led us to study the contribution of term  \eqref{penalization_term}, which is usually negligible in the Maz'ya-Shaposhnikova regime. Indeed, it works precisely as a penalization of the oscillations, with a coefficient due to the mass of the kernel at the scales between $\e$ and $1$,  forcing strong convergence in the case $\lambda=+\infty$.  This is in accordance with the separation-of-scale analysis for functionals $F_\e$, for which 
we have strong $L^2$ compactness when we let $\e\to 0$ first. The techniques used here are similar but complementary to those in \cite{BMS26}. In particular, certain Riemann--Lebesgue type lemmas suit very well the framework of singular kernels of the kind $|x-y|^{-d-2s}$. Indeed, these kernels are smooth away from the diagonal and their contributions at medium range interactions can be precisely quantified. 

Section \ref{sec:preliminaries} is devoted to the collection of several preparatory results on the average behaviour of double integrals of general interest. It should be noted that, since the kernel spreads its mass uniformly over logarithmic scales, most of the lemmas can be applied in both the $\Gamma$-liminf and the $\Gamma$-limsup inequalities. 
Section \ref{sec:liminf} contains the details of the $\Gamma$-liminf portion of Theorem \ref{thm:main}. To prove it, we employ a Fonseca--Müller type of blow up. What is new is that the contributions to reconstruct the cell formula come from two different regions of the space. Therefore, in order to recognize a competitor for the infimization problem defining $a_\lambda$, one has to perform the same modifications both at the medium- and long-range interactions level. 
Remarkably, also in the case of $a$ constant, this allows us to prove the $\Gamma$-liminf inequality without relying on a fractional Hardy inequality (see \cite{CriDLKNP23}).  
Section \ref{sec:limsup} contains the $\Gamma$-limsup argument. The techniques are similar to those of the previous section, and one can see in a more transparent way the role of each contribution (even of those that are negligible). More precisely, in the delicate term (that is, the contribution coming from \eqref{penalization_term}), we make use of the fact that $(d-1)$-dimensional sets have zero fractional capacity for $s<\frac12$ to change the boundary data and periodize the oscillations. Finally, in Section \ref{sec:pneq2} we briefly illustrate the changes for the case of general $p>1$.

\section{Notation and preliminaries}\label{sec:preliminaries}
Throughout the paper $d\geq 1$ is a positive integer. The norm of a point $x\in\Rd$ is denoted by $|x|$; for $R>0$ we set $B_R\coloneqq\{y\in\Rd:|y|<R\}$. The Lebesgue measure of a set $A\subset \Rd$ is denoted by $|A|$. We use the standard notation for Lebesgue and Sobolev spaces.

\smallskip
Let $a\colon\Rd\times\Rd\to\R$ be a Borel function that is $1$-periodic; i.e., 
$$
a(x+k,y+j)= a(x,y) \hbox{ for all }k,j\in\mathbb Z^d, x,y\in\Rd,
$$
and such that there exist $\alpha,\beta$ with $0<\alpha\leq\beta<+\infty$ such that
$$
\alpha\le a(x,y) \le \beta\hbox{ for all } x,y\in\Rd.
$$
We let $Q\coloneqq(0,1)^d$ denote the unit cell of periodicity and define the {\em contracted function}
\begin{equation}\label{def:oac}
    \overline{a}^{\rm c}(x)\coloneqq\int_Q\big(a(x,y)+a(y,x)\big)\,dy.
\end{equation}

Let $\e>0$ and let $s=s(\e)\in(0,1)$ be such that $s\to0$ as $\e\to0$; in all formulas we drop the explicit dependence of $s$ on $\e$. We assume that the limit
\begin{equation}\label{eq:scaling}
\lim_{\e\to0}\e^{-2s}=\lambda\in[1,+\infty]
\end{equation}
exists, which is not restrictive up to extraction of a subsequence.

For every $\e>0$, we define $F_\e\colon L^2(\Rd)\to[0,+\infty]$ as
\begin{equation}\label{eq:Feps}
F_\e(u)\coloneqq s\int_{\Rd}\int_{\Rd} a\Big(\frac{x}{\e},\frac{y}{\e}\Big)\frac{|u(x)-u(y)|^2}{|x-y|^{d+2s}}\,dx\,dy,
\end{equation}
with domain the fractional Sobolev space $H^s(\Rd)$.

\begin{remark}
\label{rem:PiniSantilli}
In logarithmic scale $t=\log|x-y|$, the measure $s|h|^{-d-2s}\,dh$ becomes $\sigma_{d-1}s\,e^{-2st}\,dt$, which is asymptotically uniform as $s\to0^+$. Hence, the scales $t>0$, responsible for the interactions with infinity, carry mass $\frac{\sigma_{d-1}}{2}$, while the scales $\log\e<t<0$, corresponding to the region in \eqref{penalization_term}, carry mass
\[
\frac{\sigma_{d-1}}{2}\big(\e^{-2s}-1\big)\to\frac{\sigma_{d-1}}{2}(\lambda-1).
\]
On the latter scales, the macroscopic profile is essentially frozen, $u(y)\approx u(x)$, whereas $y/\e$ is equidistributed in $Q$. Testing $F_\e$ on $u_\e(x)=u(x)\big(1+\phi(x/\e)\big)$, the corresponding energy is therefore the periodic functional of \cite{PiniSantilli2026}, weighted by $a$ and evaluated on the corrector $\phi$. For a periodic function, as in \cite{PiniSantilli2026}, the two ranges of scales play the same role, since its values at infinity are distributed as its values on the periodicity cell. The separation between the scale $\e$ and the macroscopic scale splits them.
Hence, the coefficient $a_\lambda$ in Theorem \ref{thm:main} interpolates between homogenization at infinity and homogenization on the torus.
\end{remark}

\subsection{The effective coefficients $a_\lambda$}\label{subsec:hom-coeff}
For every $\lambda\in[1,+\infty]$, every $z\in\R$ and every $\phi\in L^2(Q)$, we set
\begin{equation}\label{eq:a-lambda-z}
A_\lambda(z,\phi)\coloneqq\int_Q\overline{a}^{\rm c}(x)|z+\phi(x)|^2\,dx+(\lambda-1)\int_Q\int_Q a(x,y)|\phi(x)-\phi(y)|^2\,dx\,dy,
\end{equation}
where $\overline{a}^{\rm c}$ is defined in \eqref{def:oac}. For $\lambda=+\infty$, we use the convention $(+\infty)\cdot0=0$, so that $A_{+\infty}(z,\phi)$ is finite only if $\phi$ is constant since $a\geq\alpha>0$. We then define
\begin{equation}\label{eq:a-lambda}
a_\lambda\coloneqq\inf\Big\{A_\lambda(1,\phi):\phi\in L^2(Q)\text{ and }\int_Q\phi(x)\,dx=0\Big\};
\end{equation}
that is, formula \eqref{eq:a-lambda-0} or \eqref{eq:a-lambda-infty} in the case $\lambda=+\infty$.
The test functions for the problem defining $a_\lambda$ will be used to construct periodic recovery sequences. The function $\varphi\in L^2(Q)$ will then be identified with its $1$-periodic extension. Note that 
$$
\inf\Big\{A_\lambda(z,\phi):\phi\in L^2(Q)\text{ and }\int_Q\phi(x)\,dx=0\Big\}=a_\lambda z^2
$$
by positive $2$-homogeneity.
The definition of $a_\lambda$ allows us to define $F_\lambda\colon L^2(\Rd)\to[0,+\infty)$ as
\begin{equation}\label{eq:Flambda}
F_\lambda(u)\coloneqq\frac{\sigma_{d-1}}{2}\,a_\lambda\|u\|^2_{L^2(\Rd)},
\end{equation}
where $\sigma_{d-1}\coloneqq\Haus(\Sd)$. We note a slight abuse of notation with respect to $F_\e$.

\begin{remark}[dependence on $\lambda$]\label{rem:a-infty}
The map $\lambda\mapsto a_\lambda$ is continuous, increasing and concave (an infimum of affine functions of $\lambda$), and interpolates between the harmonic and the arithmetic mean of $\overline a^c$.

Indeed, if $\lambda=1$, then it is well known that the minimization gives the harmonic mean of $\overline a^c$; that is
$$
a_1=\Big(\int_Q \frac{1}{\overline a^c(x)}dx\Big)^{-1}.
$$
If $\lambda=+\infty$, then the only test function $\phi\in L^2(Q)$ with $A_{\infty}(1,\phi)<+\infty$ is $\phi=0$. Therefore, 
\begin{equation}\label{eq:a-infty}
a_{+\infty}=\int_Q\int_Q (a(x,y)+a(y,x))\,dx\,dy =\int_Q\overline a^c(x)\,dx=2\int_Q\int_Q a(x,y)\,dx\,dy;
\end{equation}
that is, the average of $\overline a^c$ (or twice the average of $a$).
\end{remark}

\begin{remark}[symmetric coefficients]\label{rem:symmetrx}
The replacement of $a(x,y)$ by its symmetric part $\frac12\big(a(x,y)+a(y,x)\big)$, which satisfies the same bounds and periodicity assumptions, does not change $F_\e$ nor $a_\lambda$. Hence, from now on we assume, without loss of generality, that $a$ is symmetric.
In such a case, we define 
$$\overline{a}^{y}(x)\coloneqq\int_Q a(x,y)\,dy.$$
Then we have $\overline{a}^{\rm c}=2\overline{a}^{y}$,
and we may use $2\overline{a}^{y}$ for the contracted function instead of $\overline{a}^{\rm c}$.
In this case, \eqref{eq:a-lambda-z} reads
\begin{equation*}
A_\lambda(z,\phi)=2\int_Q\overline{a}^{y}(x)|z+\phi(x)|^2\,dx+(\lambda-1)\int_Q\int_Q a(x,y)|\phi(x)-\phi(y)|^2\,dx\,dy.
\end{equation*}

\end{remark}

\begin{remark}[coefficient depending on one variable]\label{rem:separated}
We consider $a(x,y)$ of the form $\frac12\big(a_0(x)+a_0(y)\big)$, with $a_0\colon\Rd\to\R$ a $1$-periodic, bounded Borel function and satisfies $a_0\geq\alpha>0$ a.e. This corresponds (in the non-symmetric setting) to considering the coefficient $a_0(x/\e)$ alone in $F_\e$.
Then the coefficient $a_\lambda$ can be explicitly computed for every $\lambda\in[1,+\infty]$. Note that, by Remark \ref{rem:symmetrx}, this is the symmetric kernel equivalent to the choice $a(x,y)=a_0(x)$ discussed in Section \ref{sec:introduction}. By \eqref{def:oac}, the contracted function is
\begin{equation}\label{eq:separated-ac}
\overline{a}^{\rm c}(x)=a_0(x)+\overline{a}_0,\qquad\text{where }\overline{a}_0\coloneqq\int_Q a_0(\xi)\,d\xi,
\end{equation}
while, expanding the square and using that $\int_Q\phi(x)\,dx=0$, for every test function $\phi$ in \eqref{eq:a-lambda} we have
\begin{equation}\label{eq:separated-nonlocal}
\int_Q\int_Q a(x,y)|\phi(x)-\phi(y)|^2\,dx\,dy=\int_Q\overline{a}^{\rm c}(x)|\phi(x)|^2\,dx;
\end{equation}
that is, the nonlocal term in \eqref{eq:a-lambda-z} is local. Hence, if $\lambda<+\infty$, by \eqref{eq:separated-nonlocal} and after completing the square, \eqref{eq:a-lambda-z} reads
\begin{align*}
A_\lambda(1,\phi)&=\int_Q\overline{a}^{\rm c}(x)\Big(|1+\phi(x)|^2+(\lambda-1)|\phi(x)|^2\Big)dx\\
&=\lambda\int_Q\overline{a}^{\rm c}(x)\Big|\phi(x)+\frac1\lambda\Big|^2dx+\Big(1-\frac1\lambda\Big)\int_Q\overline{a}^{\rm c}(x)\,dx,
\end{align*}
so that \eqref{eq:a-lambda} reduces to minimizing the first term under the constraint $\int_Q\big(\phi(x)+\frac1\lambda\big)dx=\frac1\lambda$. This amounts to saturating a Cauchy--Schwarz inequality, which produces the unique minimizer
\begin{equation}\label{eq:separated-phi}
\phi_\lambda(x)=\frac1\lambda\bigg(\frac{a_1}{\overline{a}^{\rm c}(x)}-1\bigg),\qquad\text{where }a_1=\bigg(\int_Q\frac{dx}{\overline{a}^{\rm c}(x)}\bigg)^{-1}
\end{equation}
is the harmonic mean of $\overline{a}^{\rm c}$, and
\begin{equation}\label{eq:separated-alambda}
a_\lambda=\Big(1-\frac1\lambda\Big)\int_Q\overline{a}^{\rm c}(x)\,dx+\frac1\lambda\,a_1=\Big(1-\frac1\lambda\Big)a_{+\infty}+\frac1\lambda\,a_1,
\end{equation}
where we used \eqref{eq:a-infty} and \eqref{eq:separated-ac}. With the convention $\frac1\lambda=0$ for $\lambda=+\infty$, both \eqref{eq:separated-phi} and \eqref{eq:separated-alambda} hold for $\lambda=+\infty$ also, in accordance with \cref{rem:a-infty}.

Thus, for such kernels, $a_\lambda$ is the convex combination, with parameter $\frac1\lambda$, of the harmonic mean $a_1$ and the arithmetic mean $a_{+\infty}=2\overline{a}_0$ of the contracted function.
\end{remark}

\subsection{Asymptotics of double integrals}\label{subsec:asymptotic}
The analysis of the asymptotic behaviour of $F_\e$ will be based on the study of double integrals defined in sets of different types. To this end, in this section we gather some asymptotic estimates that could be of independent interest for variational problems on fractional Sobolev spaces.

The first result gives an estimate of the contribution when interactions are restricted between two scales.
\begin{lemma}\label{lem:annuli}
For every $s\in(0,1)$, every measurable set $A\sub\Rd$ with $|A|<+\infty$, every $0<r_1<r_2$ and every $u\in L^\infty(\Rd)$, it holds
\begin{equation}\label{eq:annuli}
s\int_A\int_{\{r_1<|x-y|<r_2\}}\frac{|u(x)-u(y)|^2}{|x-y|^{d+2s}}\,dy\,dx\leq 2\sigma_{d-1}\|u\|^2_{L^\infty(\Rd)}|A|\big(r_1^{-2s}-r_2^{-2s}\big),
\end{equation}
where $\sigma_{d-1}$ is as in \eqref{eq:Flambda}. In particular, if $(u_\e)\sub L^\infty(\Rd)$ is a bounded sequence and $r_1\coloneqq r_1(\e)$, $r_2\coloneqq r_2(\e)$ satisfy $0<r_1<r_2$ and
\begin{equation*}
\lim_{\e\to0}r_1^{-2s}=\lim_{\e\to0}r_2^{-2s}\in[0,+\infty),
\end{equation*}
then
\begin{equation}\label{eq:annu-0}
\lim_{\e\to0}s\int_A\int_{\{r_1<|x-y|<r_2\}}\frac{|u_\e(x)-u_\e(y)|^2}{|x-y|^{d+2s}}\,dy\,dx=0.
\end{equation}
\end{lemma}

\begin{proof}
Using $|u(x)-u(y)|^2\leq4\|u\|^2_{L^\infty(\Rd)}$ and integrating in polar coordinates, we obtain
\begin{equation*}
s\int_A\int_{\{r_1<|x-y|<r_2\}}\frac{|u(x)-u(y)|^2}{|x-y|^{d+2s}}\,dy\,dx\leq4s\|u\|^2_{L^\infty(\Rd)}|A|\,\sigma_{d-1}\int_{r_1}^{r_2}t^{-1-2s}\,dt,
\end{equation*}
which gives \eqref{eq:annuli}.

Applying \eqref{eq:annuli},  with $u=u_\e$, we see that the quantity in the limit in \eqref{eq:annu-0} is bounded by $2\sigma_{d-1}\sup_\e\|u_\e\|^2_{L^\infty(\Rd)}|A|\big(r_1^{-2s}-r_2^{-2s}\big)$, which vanishes as $\e\to0$.
\end{proof}

The following lemma regards the contribution of interactions in a neighbourhood of a set $\Sigma$ of finite $(d-1)$-dimensional Minkowski content.
\begin{lemma}\label{lem:jump}
Let $\Sigma\sub\Rd$ be a closed set such that there exist $C,\ell>0$ such that
\begin{equation}\label{eq:minkowski}
\big|\{x\in\Rd:\operatorname{dist}(x,\Sigma)<\tau\}\big|\leq C\tau\qquad\text{for every }\tau\in(0,\ell].
\end{equation}
Then, for every $s\in(0,1/2)$ and every $u\in L^\infty(\Rd)$, it holds
\begin{equation}\label{eq:jump}
s\iint_{\{(x,y)\in\Rd\times\Rd:\,\operatorname{dist}(x,\Sigma)\leq|x-y|<\ell\}}\frac{|u(x)-u(y)|^2}{|x-y|^{d+2s}}\,dy\,dx\leq4\sigma_{d-1}C\|u\|^2_{L^\infty(\Rd)}\,\frac{s\,\ell^{1-2s}}{1-2s}.
\end{equation}
\end{lemma}

\begin{proof}
Set $\tau(x)\coloneqq\operatorname{dist}(x,\Sigma)$ and $f(t)\coloneqq\big(t^{-2s}-\ell^{-2s}\big)_+$ for $t>0$. Note that $|\Sigma|=0$ by \eqref{eq:minkowski}. For every $x\notin\Sigma$, using $|u(x)-u(y)|^2\leq4\|u\|^2_{L^\infty(\Rd)}$ and integrating in polar coordinates, we get
\begin{eqnarray*}
s\int_{\{\tau(x)\leq|x-y|<\ell\}}\frac{|u(x)-u(y)|^2}{|x-y|^{d+2s}}dy&\leq& 4\|u\|^2_{L^\infty(\Rd)}\,\sigma_{d-1}s\int_{\tau(x)}^{\ell}t^{-1-2s}\,dt
\\&=&2\sigma_{d-1}\|u\|^2_{L^\infty(\Rd)}f\big(\tau(x)\big),
\end{eqnarray*}
where the integral is understood to be $0$ if $\tau(x)\geq\ell$. Since $f$ is decreasing on $(0,\ell]$ and vanishes on $[\ell,+\infty)$, by Cavalieri's principle and \eqref{eq:minkowski} we have
\begin{align*}
\int_{\Rd}f\big(\tau(x)\big)\,dx&=\int_0^{+\infty}\big|\{x\in\Rd:\tau(x)<f^{-1}(\xi)\}\big|\,d\xi\\
&\leq C\int_0^{+\infty}f^{-1}(\xi)\,d\xi\\
&=C\int_0^\ell f(t)\,dt=2C\,\frac{s\,\ell^{1-2s}}{1-2s},
\end{align*}
where $f^{-1}$ denotes the inverse of $f$ restricted to $(0,\ell)$. Integrating the previous estimate in $x$ gives \eqref{eq:jump}.
\end{proof}

\subsection{Averaging interactions for singular kernels}\label{subsec:riemann-lebesgue}

This section is devoted to the proof of two Riemann--Lebesgue type lemmas. In some sense, they are dual to \cite[Lemma~2.5]{BMS26}, where the case $s\to 1$ is treated, and can be proven in a similar way.

Given a Borel function $b\colon\Rd\times\Rd\to\R$ which is $1$-periodic, we set
\begin{equation*}
\overline{b}\coloneqq\int_{Q}\int_{Q}b(x,y)\,dx\,dy.
\end{equation*}
For $\e>0$ and $k\in\Z^d$ we let $Q_{\e,k}\coloneqq\e\big(k+[0,1)^d\big)$ denote the cubes of the partition of $\Rd$ induced by the lattice $\e\Z^d$ and, for $A\sub\Rd$, we set $\partial_\e A\coloneqq\{x\in\Rd:\operatorname{dist}(x,\partial A)\leq\sqrt d\,\e\}$.

\begin{lemma}[Averaging interactions at scales larger than $\e$]\label{lem:averaging}
Let $b\in L^1_{\rm loc}(\Rd\times\Rd)$ be $1$-periodic.
Then there exists $C=C_d>0$ such that
\begin{multline}\label{eq:averaging}
\bigg|s\int_A\int_{\{c\e<|x-y|<\ell\}}\frac{b\big(\frac{x}{\e},\frac{y}{\e}\big)}{|x-y|^{d+2s}}\,dy\,dx-\overline{b}\,\frac{\sigma_{d-1}}{2}\big((c\e)^{-2s}-\ell^{-2s}\big)|A|\bigg|\\
\leq C\|b\|_{L^1(Q^2)}\,(c\e)^{-2s}\Big(s\big(1+c^{-1}\big)^{d+4}|A|+\big(1+c^{-1}\big)^{d}|\partial_{2\e}A|\Big)
\end{multline}
for every $s\in(0,1)$, every $c,\ell>0$, every $\e>0$ such that $c\e<\ell$ and $4\sqrt d\,\e\leq\ell$, and every Borel set $A\sub\Rd$ with $|A|<+\infty$.
\end{lemma}

\begin{proof}
In the following, $C=C_d>0$ may change from line to line. Set
\begin{equation*}
\kappa(x,y)\coloneqq\frac{s}{|x-y|^{d+2s}},\qquad\Lambda\coloneqq\big(A\times\Rd\big)\cap\{c\e<|x-y|<\ell\}.
\end{equation*}
Integrating in polar coordinates, we get
\begin{equation}\label{eq:polar-kappa}
\int_\Lambda\kappa(x,y)\,dy\,dx=\frac{\sigma_{d-1}}{2}\big((c\e)^{-2s}-\ell^{-2s}\big)|A|,
\end{equation}
so that the left-hand side of \eqref{eq:averaging} equals the absolute value of
\begin{equation*}
\int_\Lambda\Big(b\Big(\frac{x}{\e},\frac{y}{\e}\Big)-\overline{b}\Big)\kappa(x,y)\,dy\,dx.
\end{equation*}
We cover $\Lambda$ by the sets
\begin{align*}
\Lambda_{1,\e}&\coloneqq\bigcup\big\{Q_{\e,j}\times Q_{\e,k}:Q_{\e,j}\times Q_{\e,k}\sub\Lambda,\ Q_{\e,j}\sub A\big\},\\
\Lambda_{2,\e}&\coloneqq\Lambda\cap\big(\partial_\e A\times\Rd\big),\\
\Lambda_{3,\e}&\coloneqq\Lambda\setminus\big(\Lambda_{1,\e}\cup\Lambda_{2,\e}\big).
\end{align*} By the $1$-periodicity of $b$, for all $j,k\in\Z^d$ we have
\begin{equation}\label{eq:cube-L1}
\int_{Q_{\e,j}}\int_{Q_{\e,k}}\Big|b\Big(\frac{x}{\e},\frac{y}{\e}\Big)-\overline{b}\Big|\,dy\,dx=\e^{2d}\,\|b-\overline{b}\|_{L^1(Q^2)}\leq2\e^{2d}\|b\|_{L^1(Q^2)}.
\end{equation}
Since the sets $\Lambda_{2,\e}$ and $\Lambda_{3,\e}$ are not unions of products of cubes, we estimate the corresponding contributions by covering them with products of cubes, and using \eqref{eq:cube-L1} on each of them.

Let $Q_{\e,j}\times Q_{\e,k}\sub\Lambda_{1,\e}$, fix $(\overline x,\overline y)\in Q_{\e,j}\times Q_{\e,k}$, and set $$m_{jk}\coloneqq\inf\{|x-y|^{-d-1-2s}:(x,y)\in Q_{\e,j}\times Q_{\e,k}\}.$$ Since $b$ is $1$-periodic, we have
\begin{equation*}
\int_{Q_{\e,j}}\int_{Q_{\e,k}}\Big(b\Big(\frac{x}{\e},\frac{y}{\e}\Big)-\overline{b}\Big)\kappa(x,y)\,dy\,dx=\int_{Q_{\e,j}}\int_{Q_{\e,k}}\Big(b\Big(\frac{x}{\e},\frac{y}{\e}\Big)-\overline{b}\Big)\big(\kappa(x,y)-\kappa(\overline x,\overline y)\big)\,dy\,dx.
\end{equation*}
For $x\neq y$ we have $|\nabla_{(x,y)}\kappa(x,y)|\leq\sqrt2(d+2)s|x-y|^{-d-1-2s}$. Moreover, any two points $(x,y),(x',y')\in Q_{\e,j}\times Q_{\e,k}$ satisfy $|x'-y'|\geq c\e$ and
\begin{equation*}
|x-y|\leq|x'-y'|+2\sqrt d\,\e\leq\big(1+2\sqrt d\,c^{-1}\big)|x'-y'|\leq2\sqrt d\big(1+c^{-1}\big)|x'-y'|,
\end{equation*}
so that, by the mean-value theorem,
\begin{equation*}
|\kappa(x,y)-\kappa(\overline x,\overline y)|\leq C\big(1+c^{-1}\big)^{d+3}\e s\,m_{jk}\qquad\text{for }(x,y)\in Q_{\e,j}\times Q_{\e,k},
\end{equation*}
where we used that $d+1+2s<d+3$ since $s<1$. Hence, by \eqref{eq:cube-L1} and using $\e^{2d}m_{jk}\leq\int_{Q_{\e,j}}\int_{Q_{\e,k}}|x-y|^{-d-1-2s}\,dy\,dx$, summing over the cubes of $\Lambda_{1,\e}$ and integrating in polar coordinates we obtain
\begin{align*}
\bigg|\int_{\Lambda_{1,\e}}\Big(b\Big(\frac{x}{\e},\frac{y}{\e}\Big)-\overline{b}\Big)\kappa(x,y)\,dy\,dx\bigg|&\leq C\big(1+c^{-1}\big)^{d+3}\e s\|b\|_{L^1(Q^2)}|A|\,\sigma_{d-1}\int_{c\e}^{+\infty}t^{-2-2s}\,dt\\
&\leq C\big(1+c^{-1}\big)^{d+4}\|b\|_{L^1(Q^2)}\,s\,(c\e)^{-2s}|A|.
\end{align*}

Let now $Q_{\e,j}\times Q_{\e,k}$ intersect $\Lambda_{2,\e}\cup\Lambda_{3,\e}$. Since $\Lambda\sub\{|x-y|>c\e\}$ and $\big||x-y|-|x'-y'|\big|\leq2\sqrt d\,\e$ for $(x,y),(x',y')\in Q_{\e,j}\times Q_{\e,k}$, we have
\begin{equation}\label{eq:sup-kappa}
\sup_{(Q_{\e,j}\times Q_{\e,k})\cap\Lambda}\kappa\leq s\,\max\big\{c\e,|x-y|-2\sqrt d\,\e\big\}^{-d-2s}\qquad\text{for every }(x,y)\in Q_{\e,j}\times Q_{\e,k}.
\end{equation}

If $Q_{\e,j}\times Q_{\e,k}$ intersects $\Lambda_{2,\e}$, then $Q_{\e,j}$ intersects $\partial_\e A$, so that $Q_{\e,j}\sub\partial_{2\e}A$; moreover, $|x-y|<\ell+2\sqrt d\,\e$ on $Q_{\e,j}\times Q_{\e,k}$. Hence, by \eqref{eq:cube-L1} and \eqref{eq:sup-kappa}, the contribution of $\Lambda_{2,\e}$ is at most
\begin{equation*}
2\|b\|_{L^1(Q^2)}\int_{\partial_{2\e}A}\int_{\{|x-y|<\ell+2\sqrt d\,\e\}}s\,\max\big\{c\e,|x-y|-2\sqrt d\,\e\big\}^{-d-2s}\,dy\,dx.
\end{equation*}
Integrating in polar coordinates, and using that $t\leq(1+2\sqrt d\,c^{-1})(t-2\sqrt d\,\e)$ for $t\geq(c+2\sqrt d)\e$, the inner integral is bounded by
\begin{equation*}
\sigma_{d-1}s\,(c\e)^{-d-2s}\frac{\big((c+2\sqrt d)\e\big)^d}{d}+\big(1+2\sqrt d\,c^{-1}\big)^{d-1}\sigma_{d-1}s\int_{c\e}^{+\infty}t^{-1-2s}\,dt\leq C\big(1+c^{-1}\big)^{d}(c\e)^{-2s},
\end{equation*}
so that the contribution of $\Lambda_{2,\e}$ is bounded by $C(1+c^{-1})^d\|b\|_{L^1(Q^2)}(c\e)^{-2s}|\partial_{2\e}A|$.

If $Q_{\e,j}\times Q_{\e,k}$ intersects $\Lambda_{3,\e}$ at a point $(x,y)$, then $x\in A$ and $\operatorname{dist}(x,\partial A)>\sqrt d\,\e$, so that $Q_{\e,j}\sub A$. Since $(x,y)\notin\Lambda_{1,\e}$, the set $Q_{\e,j}\times Q_{\e,k}$ contains a pair $(x',y')$ with $|x'-y'|\leq c\e$ or $|x'-y'|\geq\ell$. As $\big||x-y|-|x'-y'|\big|\leq2\sqrt d\,\e$ and $|x-y|<\ell$, we get that either $|x-y|\leq(c+2\sqrt d)\e$ on $Q_{\e,j}\times Q_{\e,k}$, or $\ell-2\sqrt d\,\e\leq|x-y|\leq\ell+2\sqrt d\,\e$ on $Q_{\e,j}\times Q_{\e,k}$. In the first case, the supremum in \eqref{eq:sup-kappa} is at most $s(c\e)^{-d-2s}$; in the second case, since $\ell-2\sqrt d\,\e\geq\ell/2$, it is at most $s(\ell/2)^{-d-2s}$. Hence, by \eqref{eq:cube-L1}, the contribution of $\Lambda_{3,\e}$ is at most
\begin{align*}
&2\|b\|_{L^1(Q^2)}|A|\Big(s(c\e)^{-d-2s}\,\omega_d\big((c+2\sqrt d)\e\big)^d+s(\ell/2)^{-d-2s}\,\sigma_{d-1}\,4\sqrt d\,\e\,(2\ell)^{d-1}\Big)\\
&\qquad\leq C\|b\|_{L^1(Q^2)}\,s\Big(\big(1+c^{-1}\big)^d(c\e)^{-2s}+\frac{\e}{\ell}\ell^{-2s}\Big)|A|,
\end{align*}
where $\omega_d$ denotes the volume of the unit ball. Since $\e\leq\ell$ and $\ell^{-2s}\leq(c\e)^{-2s}$, the right-hand side is bounded by $C(1+c^{-1})^{d+4}\|b\|_{L^1(Q^2)}\,s\,(c\e)^{-2s}|A|$, and \eqref{eq:averaging} follows.
\end{proof}

\begin{obs}\label{obs:averaging-Linfty}
Note that, if $b\in L^\infty(Q^2)$, since $\|b\|_{L^1(Q^2)}\leq\|b\|_{L^\infty(Q^2)}$, estimate \eqref{eq:averaging} holds a fortiori with $\|b\|_{L^\infty(Q^2)}$ in place of $\|b\|_{L^1(Q^2)}$.
\end{obs}

\begin{lemma}[Averaging long-range interactions]\label{lem:long-range}
Let $s\in(0,1)$ and $\e>0$. Let $\overline{a}^{y}$ be as in Remark \ref{rem:symmetrx}. Then there exists $C=C_d>0$ such that
\begin{multline}\label{eq:long-range}
\bigg|s\int_A\int_{\Rd\setminus B_R}a\Big(\frac{x}{\e},\frac{y}{\e}\Big)\frac{|u(x)|^2}{|x-y|^{d+2s}}\,dy\,dx-\frac{\sigma_{d-1}}{2}R^{-2s}\int_A\overline{a}^{y}\Big(\frac{x}{\e}\Big)|u(x)|^2\,dx\bigg|\\
\leq C\beta\,s\,\frac{\rho+\e}{R}R^{-2s}\|u\|^2_{L^2(A)}
\end{multline}
for every $\rho>0$, every Borel set $A\sub B_\rho$, every $u\in L^2(A)$ and every $R\geq\max\{2\rho,\sqrt d\,\e\}$.
\end{lemma}

\begin{proof}
In the following, $C=C_d>0$ may change from line to line. Set $\kappa(y)\coloneqq s|y|^{-d-2s}$ and
\begin{equation*}
J\coloneqq s\int_A|u(x)|^2\int_{\Rd\setminus B_R}a\Big(\frac{x}{\e},\frac{y}{\e}\Big)\frac{1}{|y|^{d+2s}}\,dy\,dx.
\end{equation*}

We first replace $|x-y|$ with $|y|$. Let $x\in A$ and $|y|\geq R$. For every $\theta\in[0,1]$ we have $|y-\theta x|\geq|y|-\rho\geq|y|/2$, since $R\geq2\rho$. Hence, by the mean-value theorem applied to $z\mapsto|z|^{-d-2s}$ on the segment joining $y-x$ and $y$,
\begin{equation*}
\Big||x-y|^{-d-2s}-|y|^{-d-2s}\Big|\leq(d+2)\,\rho\,\big(|y|/2\big)^{-d-1-2s}\leq C\rho|y|^{-d-1-2s},
\end{equation*}
where we used again that $d+1+2s<d+3$ since $s<1$. Using $a\leq\beta$ and integrating in polar coordinates, we get
\begin{equation}\label{eq:long-range-step1}
\begin{aligned}
\bigg|s\int_A\int_{\Rd\setminus B_R}a\Big(\frac{x}{\e},\frac{y}{\e}\Big)\frac{|u(x)|^2}{|x-y|^{d+2s}}\,dy\,dx-J\bigg|&\leq C\beta s\rho\|u\|^2_{L^2(A)}\sigma_{d-1}\int_R^{+\infty}t^{-2-2s}\,dt\\
&\leq C\beta\,s\,\frac{\rho}{R}R^{-2s}\|u\|^2_{L^2(A)}.
\end{aligned}
\end{equation}

We now average in $y$. Fix $x\in A$ and set $g(y)\coloneqq a(x/\e,y/\e)$. Integrating in polar coordinates, we have
\begin{equation*}
\int_{\Rd\setminus B_R}\kappa(y)\,dy=\frac{\sigma_{d-1}}{2}R^{-2s},
\end{equation*}
so that
\begin{equation*}
\int_{\Rd\setminus B_R}g(y)\kappa(y)\,dy-\frac{\sigma_{d-1}}{2}R^{-2s}\,\overline{a}^{y}\Big(\frac{x}{\e}\Big)=\int_{\Rd\setminus B_R}\Big(g(y)-\overline{a}^{y}\Big(\frac{x}{\e}\Big)\Big)\kappa(y)\,dy.
\end{equation*}
Note that $|g-\overline{a}^{y}(x/\e)|\leq\beta$. We partition $\Rd\setminus B_R$ into the sets
\begin{equation*}
\Gamma_{1,\e}\coloneqq\bigcup\big\{Q_{\e,k}:Q_{\e,k}\sub\Rd\setminus B_R\big\},\qquad\Gamma_{2,\e}\coloneqq\big(\Rd\setminus B_R\big)\setminus\Gamma_{1,\e}.
\end{equation*}
Let $Q_{\e,k}\sub\Gamma_{1,\e}$ and fix $\overline y\in Q_{\e,k}$. Since $a$ is $1$-periodic in the second variable, we have
\begin{equation*}
\int_{Q_{\e,k}}g(y)\,dy=\e^d\,\overline{a}^{y}\Big(\frac{x}{\e}\Big),
\end{equation*}
so that
\begin{equation*}
\int_{Q_{\e,k}}\Big(g(y)-\overline{a}^{y}\Big(\frac{x}{\e}\Big)\Big)\kappa(y)\,dy=\int_{Q_{\e,k}}\Big(g(y)-\overline{a}^{y}\Big(\frac{x}{\e}\Big)\Big)\big(\kappa(y)-\kappa(\overline y)\big)\,dy.
\end{equation*}
For $y\neq0$ we have $|\nabla\kappa(y)|\leq(d+2)s|y|^{-d-1-2s}$. Moreover, any two points $y,y'\in Q_{\e,k}$ satisfy $|y'|\geq R\geq\sqrt d\,\e$ and $|y|\leq|y'|+\sqrt d\,\e\leq2|y'|$. Hence, by the mean-value theorem,
\begin{equation*}
|\kappa(y)-\kappa(\overline y)|\leq C\e s|y|^{-d-1-2s}\qquad\text{for }y\in Q_{\e,k},
\end{equation*}
where we used again that $d+1+2s<d+3$ since $s<1$. Summing over the cubes of $\Gamma_{1,\e}$ and integrating in polar coordinates, we obtain
\begin{equation*}
\bigg|\int_{\Gamma_{1,\e}}\Big(g(y)-\overline{a}^{y}\Big(\frac{x}{\e}\Big)\Big)\kappa(y)\,dy\bigg|\leq C\beta\e s\,\sigma_{d-1}\int_R^{+\infty}t^{-2-2s}\,dt\leq C\beta\,s\,\frac{\e}{R}R^{-2s}.
\end{equation*}
If $y\in\Gamma_{2,\e}$, the cube $Q_{\e,k}$ containing $y$ intersects $B_R$, so that $|y|<R+\sqrt d\,\e$. Hence, integrating in polar coordinates,
\begin{equation*}
\int_{\Gamma_{2,\e}}\kappa(y)\,dy\leq\sigma_{d-1}s\int_R^{R+\sqrt d\,\e}t^{-1-2s}\,dt\leq\sqrt d\,\sigma_{d-1}s\,\e R^{-1-2s}=C s\,\frac{\e}{R}R^{-2s}.
\end{equation*}
Using $|g-\overline{a}^{y}(x/\e)|\leq\beta$ once more, we conclude that
\begin{equation*}
\bigg|\int_{\Rd\setminus B_R}a\Big(\frac{x}{\e},\frac{y}{\e}\Big)\kappa(y)\,dy-\frac{\sigma_{d-1}}{2}R^{-2s}\,\overline{a}^{y}\Big(\frac{x}{\e}\Big)\bigg|\leq C\beta\,s\,\frac{\e}{R}R^{-2s}
\end{equation*}
for every $x\in A$. Multiplying by $|u(x)|^2$ and integrating over $A$, we get
\begin{equation*}
\bigg|J-\frac{\sigma_{d-1}}{2}R^{-2s}\int_A\overline{a}^{y}\Big(\frac{x}{\e}\Big)|u(x)|^2\,dx\bigg|\leq C\beta\,s\,\frac{\e}{R}R^{-2s}\|u\|^2_{L^2(A)},
\end{equation*}
which, together with \eqref{eq:long-range-step1}, gives \eqref{eq:long-range}.
\end{proof}

\section{Liminf inequality}\label{sec:liminf}
In this section we prove that for all $u\in L^2(\Rd)$ and for all sequences $u_\e\in L^2(\Rd)$ weakly converging to $u$ in $L^2(\Rd)$ we have
$$
\liminf_{\e\to 0}F_\e(u_\e)\ge F_\lambda(u).
$$
To this end, we use a variant of the Fonseca--M\"uller blow-up method \cite{FM,BMS}. This allows us to reduce to competitors defined on small cubes and converging to a constant value. From these, a periodization-by-convexity argument gives a lower bound in terms of $F_\lambda(u)$ by comparison with the problems defining $a_\lambda$. 

\subsection{Localization, truncation and blow-up}\label{subsec:bis-blowup}

Let $u\in L^2(\Rd)$ and let $(u_\e)\sub L^2(\Rd)$ be such that $u_\e$
converges weakly to $u$ in $L^2(\Rd)$, and $F_\e(u_\e)\leq E$ for every $\e>0$ and some fixed $E>0$. Up to subsequences, we assume that the limit $\lim_{\e\to0}F_\e(u_\e)$ exists and that $|u_\e|$ converges weakly to $v$ in $L^2(\Rd)$ for some $v\in L^2(\Rd)$. We also fix a ball $B\sub\Rd$ and a parameter $R\coloneqq R(\e)$ such that
\begin{equation}\label{eq:bis-R}
\lim_{\e\to0}R=+\infty,\qquad\lim_{\e\to0}R^{-2s}=1.
\end{equation}

\def\xim{\xi_M }
For every $M\in\N$ we set
\begin{equation}\label{eq:bis-truncation}
u^M_\e\coloneqq\max\big\{\min\{u_\e,M\},-M\big\},
\end{equation}
so that, pointwise,
\begin{equation}\label{eq:bis-truncation-properties}
|u^M_\e|\leq M,\qquad|u^M_\e(x)-u^M_\e(y)|\leq|u_\e(x)-u_\e(y)|,\qquad|u^M_\e-u_\e|\leq\frac{|u_\e|^2}{M}.
\end{equation}
Up to subsequences, we assume that for every $M\in\N$ there exists $\xim \in L^2(\R^d)$, such that $u^M_\e$ converges weakly to $\xim $ in $L^2(\Rd)$. 
By the last inequality in \eqref{eq:bis-truncation-properties} and by the lower semicontinuity of the $L^1(B)$ norm with respect to the weak convergence in $L^2(B)$,
\begin{equation*}
\|\xim -u\|_{L^1(B)}\leq\frac{1}{M}\sup_{\e>0}\|u_\e\|^2_{L^2(\Rd)},
\end{equation*}
so that $\xim \to u$ in $L^1(B)$ and, up to a subsequence, almost everywhere in $B$, as $M\to+\infty$.

For $\e$ small enough we have $B\sub B_R$, so that, by the symmetry of $a$ and by the nonnegativity of the integrand, $F_\e(u_\e)\geq\mu_{B,\e}(B)$, where $\mu_{B,\e}$ is the measure defined, for every Borel set $A\sub\Rd$, by
\begin{eqnarray}\label{eq:bis-measure}\nonumber
&&\mu_{B,\e}(A)\coloneqq2s\int_{B\cap A}\int_{\Rd\setminus B_R}a\Big(\frac{x}{\e},\frac{y}{\e}\Big)\frac{|u_\e(x)-u_\e(y)|^2}{|x-y|^{d+2s}}\,dy\,dx\\
&&\hskip3cm +s\int_{B\cap A}\int_B a\Big(\frac{x}{\e},\frac{y}{\e}\Big)\frac{|u_\e(x)-u_\e(y)|^2}{|x-y|^{d+2s}}\,dy\,dx.
\end{eqnarray}
The sequence $(\mu_{B,\e})$ is equibounded, so that, up to subsequences, $\mu_{B,\e}\rightharpoonup\mu_B$ in duality with $\C_b(\Rd)$, and, since $B$ is open,
\begin{equation}\label{eq:bis-density-integral}
\liminf_{\e\to0}F_\e(u_\e)\geq\liminf_{\e\to0}\mu_{B,\e}(B)\geq\mu_B(B)\geq\int_B\frac{d\mu_B}{d\L^d}(x)\,dx.
\end{equation}

Let $Q_\rho(x)$ denote the closed hypercube of side $\rho$ centred at $x$. By the Lebesgue differentiation theorem and by a diagonal argument, for $\L^d$-almost every $x_0\in B$ we may choose $\rho\coloneqq\rho(\e)\to0$ so that
\begin{equation}\label{eq:bis-density}
\frac{d\mu_B}{d\L^d}(x_0)=\lim_{\e\to0}\frac{\mu_{B,\e}\big(Q_\rho(x_0)\big)}{\rho^d}<+\infty.
\end{equation}
Since $\rho$ is free to tend to $0$ as slowly as desired, we fix $x_0$ to be, in addition, a Lebesgue point of $v$ and of $\xim $ for every $M\in\N$, and a point at which $\xim \to u$ as $M\to+\infty$, and we further require that
\begin{equation}\label{eq:bis-rho}
\lim_{\e\to0}\rho^{-2s}=1,\qquad\lim_{\e\to0}s\rho^{-d}=0,\qquad N\coloneqq\frac{\rho}{\e}\in\N,
\end{equation}
and that
\begin{equation}\label{eq:bis-averages}
\lim_{\e\to0}\frac{1}{\rho^d}\int_{Q_\rho(x_0)}|u_\e|\,dx=v(x_0),\qquad\lim_{\e\to0}\frac{1}{\rho^d}\int_{Q_\rho(x_0)}u^M_\e\,dx=\xim (x_0)\quad\text{for every }M\in\N.
\end{equation}
From now on we drop the dependence on $x_0$ and simply write $Q_\rho$; note that $Q_\rho\sub B$ for $\e$ small enough, since $B$ is open. By \eqref{eq:bis-density-integral} and \eqref{eq:bis-density}, the $\Gamma$-liminf inequality follows once we produce a lower bound for
\begin{align}
\frac{\mu_{B,\e}\big(Q_\rho\big)}{\rho^d}&=\frac{2s}{\rho^d}\int_{Q_\rho}\int_{\Rd\setminus B_R}a\Big(\frac{x}{\e},\frac{y}{\e}\Big)\frac{|u_\e(x)-u_\e(y)|^2}{|x-y|^{d+2s}}\,dy\,dx\nonumber\\
&\quad+\frac{s}{\rho^d}\int_{Q_\rho}\int_B a\Big(\frac{x}{\e},\frac{y}{\e}\Big)\frac{|u_\e(x)-u_\e(y)|^2}{|x-y|^{d+2s}}\,dy\,dx\eqqcolon I_{1,\e}+I_{2,\e}.\label{eq:bis-split}
\end{align}

\subsection{The periodic competitor}\label{subsec:bis-competitor}

We fix $M\in\N$ and a parameter $\delta\coloneqq\delta(\e)>0$ such that
\begin{equation}\label{eq:bis-delta}
\lim_{\e\to0}\frac{\delta}{\rho}=0,\qquad\lim_{\e\to0}\frac{\e}{\delta}=0,\qquad\lim_{\e\to0}\delta^{-2s}=1,
\end{equation}
and, for $w\in L^2(Q_\rho)$, we set
\begin{equation}\label{eq:bis-G}
G_\e(w)\coloneqq\frac{s}{\rho^d}\int_{Q_\rho}\int_{Q_\rho\cap\{|x-y|>\e\}}a\Big(\frac{x}{\e},\frac{y}{\e}\Big)\frac{|w(x)-w(y)|^2}{|x-y|^{d+2s}}\,dy\,dx,
\end{equation}
so that, discarding nonnegative contributions and using \eqref{eq:bis-truncation-properties},
\begin{equation}\label{eq:bis-I2-G}
I_{2,\e}\geq G_\e(u^M_\e).
\end{equation}

We now replace $u^M_\e$ by an $\e$-periodic function, in three steps: we flatten $u^M_\e$ near $\partial Q_\rho$, we extend it by $\rho$-periodicity, and we average its translates by $\e\Z^d$.

\begin{lemma}[Periodic competitor]\label{lem:bis-competitor}
Let $S\coloneqq Q_\rho\setminus Q_{\rho-\delta}$, where $Q_{\rho-\delta}$ denotes the closed hypercube of side $\rho-\delta$ centred at $x_0$, and let
\begin{equation*}
c_\e\coloneqq\frac{1}{|S|}\int_Su^M_\e(x)\,dx,\qquad v^M_\e\coloneqq u^M_\e\,\rchi_{Q_{\rho-\delta}}+c_\e\,\rchi_S,
\end{equation*}
extended to $\Rd$ by $\rho$-periodicity. Let $N$ be as in \eqref{eq:bis-rho} and
\begin{equation*}
w^M_\e\coloneqq\frac{1}{N^d}\sum_{k\in\{0,\dots,N-1\}^d}v^M_{\e,k},\qquad v^M_{\e,k}\coloneqq v^M_\e(\cdot+\e k).
\end{equation*}
Then $w^M_\e$ is $\e$-periodic, $\|w^M_\e\|_{L^\infty(\Rd)}\leq M$,
\begin{equation}\label{eq:bis-mean}
\frac{1}{\rho^d}\int_{Q_\rho}w^M_\e(x)\,dx=\frac{1}{\rho^d}\int_{Q_\rho}u^M_\e(x)\,dx,
\end{equation}
\begin{eqnarray}\label{eq:bis-energy}\nonumber
G_\e(w^M_\e)&\leq& G_\e(u^M_\e)+4d\,\sigma_{d-1}\beta M^2\big(\e^{-2s}-(\sqrt d\,\rho)^{-2s}\big)\frac{\delta}{\rho}\\
&&\hskip2cm+2\sigma_{d-1}\beta M^2\Big(\Big(\frac{\delta}{2}\Big)^{-2s}-(\sqrt d\,\rho)^{-2s}\Big),
\end{eqnarray}
and
\begin{equation}\label{eq:bis-long-range}
\int_{Q_\rho}\overline{a}^{y}\Big(\frac{x}{\e}\Big)|w^M_\e(x)|^2\,dx\leq\int_{Q_\rho}\overline{a}^{y}\Big(\frac{x}{\e}\Big)|u_\e(x)|^2\,dx+d\,\beta M^2\rho^{d-1}\delta.
\end{equation}
\end{lemma}

\begin{proof}
By construction $w^M_\e$ is $\e$-periodic and, since $|u^M_\e|\leq M$ and $|c_\e|\leq M$, also $|v^M_\e|\leq M$ and $|w^M_\e|\leq M$. By the definition of $c_\e$ we have $\int_{Q_\rho}v^M_\e dx=\int_{Q_\rho}u^M_\e dx$, while the $\rho$-periodicity of $v^M_\e$ gives $\int_{Q_\rho}v^M_{\e,k}=\int_{Q_\rho}v^M_\e$ for every $k\in\Z^d$; averaging over $k$ gives \eqref{eq:bis-mean}.

We prove \eqref{eq:bis-energy}. We first compare $v^M_\e$ with $u^M_\e$. Splitting the domain of integration in $G_\e(v^M_\e)$ according to whether $x$ and $y$ belong to $Q_{\rho-\delta}$ or to $S$, the contribution of $Q_{\rho-\delta}\times Q_{\rho-\delta}$ is at most $G_\e(u^M_\e)$, the contribution of $S\times S$ vanishes because $v^M_\e$ is constant on $S$, and, using $|v^M_\e(x)-v^M_\e(y)|\leq2M$, $a\leq\beta$ and $|S|\leq d\,\rho^{d-1}\delta$, the contribution of the remaining pairs is at most
\begin{equation*}
\frac{8\beta M^2s}{\rho^d}|S|\,\sigma_{d-1}\int_{\e}^{\sqrt d\,\rho}t^{-1-2s}\,dt\leq4d\,\sigma_{d-1}\beta M^2\big(\e^{-2s}-(\sqrt d\,\rho)^{-2s}\big)\frac{\delta}{\rho}.
\end{equation*}

We then compare $w^M_\e$ with $v^M_\e$. Since $a$ is $1$-periodic, the integrand of $G_\e$ is invariant under the simultaneous translation $(x,y)\mapsto(x+\e k,y+\e k)$; hence, setting $Q'\coloneqq Q_\rho+\e k$ and assuming without loss of generality that $\e k\in[0,\rho)^d$, we obtain
\begin{equation*}
G_\e(v^M_{\e,k})=\frac{s}{\rho^d}\int_{Q'}\int_{Q'\cap\{|x-y|>\e\}}a\Big(\frac{x}{\e},\frac{y}{\e}\Big)\frac{|v^M_\e(x)-v^M_\e(y)|^2}{|x-y|^{d+2s}}\,dy\,dx.
\end{equation*}
We write $Q'=\bigcup_{\sigma\in\{0,1\}^d}A_\sigma$ with $A_\sigma\coloneqq Q'\cap(Q_\rho+\rho\sigma)$; the sets $A_\sigma-\rho\sigma$ are then pairwise disjoint subsets of $Q_\rho$ covering it, and we split the double integral accordingly. Since $\rho/\e\in\N$ (recall \eqref{eq:bis-rho}) and $v^M_\e$ is $\rho$-periodic, the integrand is invariant under $(x,y)\mapsto(x+\rho\sigma,y+\rho\sigma)$, so that the contributions of the sets $A_\sigma\times A_\sigma$ add up to at most $G_\e(v^M_\e)$. If instead $\sigma\neq\tau$, $x\in A_\sigma$ and $y\in A_\tau$, the segment joining $x$ and $y$ meets $\bigcup_{j\in\Z^d}\partial(Q_\rho+\rho j)$, on whose $\delta/2$-neighbourhood $v^M_\e$ is constant; hence $v^M_\e(x)=v^M_\e(y)$ whenever $|x-y|<\delta/2$, and the contributions of the sets $A_\sigma\times A_\tau$ with $\sigma\neq\tau$ add up to at most
\begin{equation*}
\frac{4\beta M^2s}{\rho^d}\int_{Q'}\int_{\{\delta/2\leq|x-y|\leq\sqrt d\,\rho\}}\frac{dy\,dx}{|x-y|^{d+2s}}=2\sigma_{d-1}\beta M^2\Big(\Big(\frac{\delta}{2}\Big)^{-2s}-(\sqrt d\,\rho)^{-2s}\Big).
\end{equation*}
Since $G_\e$ is convex, averaging over $k\in\{0,\dots,N-1\}^d$ gives \eqref{eq:bis-energy}.

Finally we prove \eqref{eq:bis-long-range}. Since $|u^M_\e|\leq|u_\e|$ pointwise, it is enough to prove it with $u^M_\e$ in place of $u_\e$. As $v^M_\e=u^M_\e$ on $Q_{\rho-\delta}$ and $|v^M_\e|\leq M$, using $\overline{a}^{y}\leq\beta$ and $|S|\leq d\,\rho^{d-1}\delta$ we get
\begin{equation*}
\int_{Q_\rho}\overline{a}^{y}\Big(\frac{x}{\e}\Big)|v^M_\e(x)|^2\,dx\leq\int_{Q_\rho}\overline{a}^{y}\Big(\frac{x}{\e}\Big)|u^M_\e(x)|^2\,dx+d\,\beta M^2\rho^{d-1}\delta.
\end{equation*}
Moreover $\overline{a}^{y}$ is $1$-periodic, so that, arguing as above but without the contributions of the sets $A_\sigma\times A_\tau$ with $\sigma\neq\tau$, we obtain $\int_{Q_\rho}\overline{a}^{y}(x/\e)|v^M_{\e,k}|^2\,dx=\int_{Q_\rho}\overline{a}^{y}(x/\e)|v^M_\e|^2\,dx$ for every $k\in\Z^d$; the convexity of $t\mapsto t^2$ concludes the proof.
\end{proof}

\subsection{Lower bounds in the cell}\label{subsec:bis-cell}

We set
\begin{equation}\label{eq:bis-cell-functions}
\nu_\e\coloneqq\frac{1}{\rho^d}\int_{Q_\rho}w^M_\e(x)\,dx,\qquad\phi_\e(z)\coloneqq w^M_\e(\e z)-\nu_\e,
\end{equation}
so that $\phi_\e$ is $1$-periodic with zero mean on $Q$, $|\nu_\e|\leq M$ and $\|\phi_\e\|_{L^\infty(\Rd)}\leq2M$. By \eqref{eq:bis-mean} and \eqref{eq:bis-averages},
\begin{equation}\label{eq:bis-z-limit}
\lim_{\e\to0}\nu_\e=\xim (x_0).
\end{equation}
Moreover, since $x\mapsto\overline{a}^{y}(x/\e)|w^M_\e(x)|^2$ is $\e$-periodic and $Q_\rho$ is the union of $N^d$ periodicity cells,
\begin{equation}\label{eq:bis-cell-identity}
\frac{1}{\rho^d}\int_{Q_\rho}\overline{a}^{y}\Big(\frac{x}{\e}\Big)|w^M_\e(x)|^2\,dx=\int_Q\overline{a}^{y}(z)|\nu_\e+\phi_\e(z)|^2\,dz.
\end{equation}

\emph{Long-range interactions.} Expanding the square in $I_{1,\e}$, see \eqref{eq:bis-split}, and dropping the nonnegative term containing $|u_\e(y)|^2$, we get
\begin{equation*}
I_{1,\e}\geq\frac{2s}{\rho^d}\int_{Q_\rho}\int_{\Rd\setminus B_R}a\Big(\frac{x}{\e},\frac{y}{\e}\Big)\frac{|u_\e(x)|^2}{|x-y|^{d+2s}}\,dy\,dx-J_\e,
\end{equation*}
where
\begin{equation}\label{eq:bis-J}
J_\e\coloneqq\frac{4s}{\rho^d}\int_{Q_\rho}\int_{\Rd\setminus B_R}a\Big(\frac{x}{\e},\frac{y}{\e}\Big)\frac{u_\e(x)u_\e(y)}{|x-y|^{d+2s}}\,dy\,dx.
\end{equation} Let $r>0$ be a fixed radius such that $B\sub B_r$. For $\e$ small enough we have $R\geq2r$, so that $|x-y|\geq|y|/2$ for every $x\in B$ and $y\in\Rd\setminus B_R$ and, integrating in polar coordinates, one obtains
\begin{equation*}
\int_{\Rd\setminus B_R}\frac{dy}{|x-y|^{2d+4s}}\leq C_d\,R^{-d-4s}\qquad\text{for every }x\in B.
\end{equation*}
Indeed, unlike its $L^1$ norm, the $L^2$ norm of the kernel vanishes as $R\to+\infty$. Hence, by $a\leq\beta$, applying the Cauchy--Schwarz inequality in $y$, and since the resulting bound is uniform in $x$, it follows that
\begin{equation*}
|J_\e|\leq C_d\,\beta\,s\,R^{-\frac d2-2s}\|u_\e\|_{L^2(\Rd)}\,\frac{1}{\rho^d}\int_{Q_\rho}|u_\e(x)|\,dx,
\end{equation*}
which vanishes as $\e\to0$ by \eqref{eq:bis-averages}, since $s\to0$ and $R\to+\infty$. Applying \cref{lem:long-range} with $A=Q_\rho\sub B_r$ and $u=u_\e$, which is admissible for $\e$ small enough, and then \eqref{eq:bis-long-range} and \eqref{eq:bis-cell-identity}, we conclude that
\begin{equation}\label{eq:bis-I1}
I_{1,\e}\geq\sigma_{d-1}R^{-2s}\int_Q\overline{a}^{y}(z)|\nu_\e+\phi_\e(z)|^2\,dz-\sigma_{d-1}\beta\,d\,M^2\frac{\delta}{\rho}-J_\e-K_\e,
\end{equation}
where $K_\e\coloneqq\frac{2C_d\beta}{\rho^d}s\frac{r+\e}{R}R^{-2s}\|u_\e\|^2_{L^2(\Rd)}$ vanishes as $\e\to0$ by \eqref{eq:bis-rho}.

\emph{Short-range interactions.} We set $b_\e(x,y)\coloneqq a(x,y)|\phi_\e(x)-\phi_\e(y)|^2$, which is a $1$-periodic function with $0\leq b_\e\leq16\beta M^2$ and
\begin{equation}\label{eq:bis-b-eps}
\|b_\e\|_{L^1(Q^2)}=\overline{b_\e}=\int_Q\int_Q a(x,y)|\phi_\e(x)-\phi_\e(y)|^2\,dx\,dy\geq2\alpha\|\phi_\e\|^2_{L^2(Q)},
\end{equation}
where the last inequality follows from $a\geq\alpha$ and from $\int_Q\phi_\e=0$. Since $w^M_\e(x)-w^M_\e(y)=\phi_\e(x/\e)-\phi_\e(y/\e)$, and since $y\in Q_\rho$ whenever $x\in Q_{\rho-2\delta}$ and $|x-y|<\delta$, discarding nonnegative contributions in $G_\e(w^M_\e)$ and applying \cref{lem:averaging} with $b=b_\e$, $A=Q_{\rho-2\delta}$, $c=1$ and outer radius $\delta$, we get
\begin{equation}\label{eq:bis-Geps}
G_\e(w^M_\e)\geq\frac{\sigma_{d-1}}{2}\big(\e^{-2s}-\delta^{-2s}\big)\,\overline{b_\e}\,\frac{|Q_{\rho-2\delta}|}{\rho^d}-C_d\,\overline{b_\e}\,\e^{-2s}\Big(2^{d+4}s+2^d\,\frac{|\partial_{2\e}Q_{\rho-2\delta}|}{\rho^d}\Big).
\end{equation}
Since $|Q_{\rho-2\delta}|/\rho^d\to1$, $|\partial_{2\e}Q_{\rho-2\delta}|\leq C_d\rho^{d-1}\e$ and $\overline{b_\e}\leq16\beta M^2$, combining \eqref{eq:bis-I2-G}, \eqref{eq:bis-energy} and \eqref{eq:bis-Geps} we obtain
\begin{eqnarray}\label{eq:bis-I2}\nonumber
I_{2,\e}&\geq&\frac{\sigma_{d-1}}{2}\big(\e^{-2s}-\delta^{-2s}\big)\,\overline{b_\e}\,\big(1+o(1)\big)\\
&&-C\,M^2\bigg(\e^{-2s}\Big(s+\frac{\e}{\rho}+\frac{\delta}{\rho}\Big)+\Big(\frac{\delta}{2}\Big)^{-2s}-(\sqrt d\,\rho)^{-2s}\bigg),
\end{eqnarray}
where $C=C(d,\alpha,\beta)$.

\subsection{Conclusion}\label{subsec:bis-conclusion}

Now, we divide the two cases $\lambda<+\infty$ and $\lambda=+\infty$, which are treated separately.

\emph{The case $\lambda<+\infty$.} Here $\e^{-2s}\to\lambda$ is bounded and $\e^{-2s}-\delta^{-2s}\to\lambda-1$, so that the error term in \eqref{eq:bis-I2} vanishes by \eqref{eq:bis-rho} and \eqref{eq:bis-delta}, and
\begin{equation*}
I_{2,\e}\geq\frac{\sigma_{d-1}}{2}(\lambda-1)\int_Q\int_Q a(x,y)|\phi_\e(x)-\phi_\e(y)|^2\,dx\,dy+o(1)\qquad\text{as }\e\to0.
\end{equation*}
Gathering this estimate with \eqref{eq:bis-split} and \eqref{eq:bis-I1}, and recalling \eqref{eq:a-lambda-z} and that $R^{-2s}\to1$, we get
\begin{equation}\label{eq:bis-cell-recognized}
\frac{\mu_{B,\e}\big(Q_\rho\big)}{\rho^d}\geq\frac{\sigma_{d-1}}{2}\,A_\lambda(\nu_\e,\phi_\e)+o(1)\qquad\text{as }\e\to0.
\end{equation}

\emph{The case $\lambda=+\infty$.} Here we discard $I_{2,\e}\geq0$ in \eqref{eq:bis-split}, and we only use \eqref{eq:bis-I2} to rule out the oscillations of $w^M_\e$. Indeed, $\e^{2s}\delta^{-2s}\to0$, so that $\e^{2s}\big(\e^{-2s}-\delta^{-2s}\big)\to1$; multiplying \eqref{eq:bis-I2} by $\e^{2s}$, and using that $\e^{2s}\big(\e^{-2s}-\delta^{-2s}\big)\leq1$ and that $I_{2,\e}\leq\mu_{B,\e}(Q_\rho)/\rho^d$ is bounded by \eqref{eq:bis-density}, we obtain
\begin{equation*}
\frac{\sigma_{d-1}}{2}\,\overline{b_\e}\,\big(1+o(1)\big)\leq\e^{2s}I_{2,\e}+C\,M^2\Big(s+\frac{\e}{\rho}+\frac{\delta}{\rho}+\e^{2s}\Big)\longrightarrow0,
\end{equation*}
therefore $\overline{b_\e}\to0$ and, by \eqref{eq:bis-b-eps}, $\phi_\e\to0$ in $L^2(Q)$. Since $|\nu_\e|\leq M$ and $\overline{a}^{y}\leq\beta$, this gives
\begin{equation*}
\int_Q\overline{a}^{y}(z)|\nu_\e+\phi_\e(z)|^2\,dz=|\nu_\e|^2\int_Q\overline{a}^{y}(z)\,dz+o(1)=\overline{a}\,|\nu_\e|^2+o(1),
\end{equation*}
so that, by \eqref{eq:bis-split} and \eqref{eq:bis-I1},
\begin{equation*}
\frac{\mu_{B,\e}\big(Q_\rho\big)}{\rho^d}\geq\sigma_{d-1}\,\overline{a}\,|\nu_\e|^2+o(1)=\frac{\sigma_{d-1}}{2}\,a_{+\infty}|\nu_\e|^2+o(1)\qquad\text{as }\e\to0,
\end{equation*}
where we used \eqref{eq:a-infty}.

\emph{Conclusion.} In both cases, since $A_\lambda(z,z\phi)=z^2A_\lambda(1,\phi)$, we have $\inf_\phi A_\lambda(z,\phi)=a_\lambda z^2$, so that
\begin{equation*}
\frac{\mu_{B,\e}\big(Q_\rho\big)}{\rho^d}\geq\frac{\sigma_{d-1}}{2}\,a_\lambda|\nu_\e|^2+o(1)\qquad\text{as }\e\to0.
\end{equation*}
Hence, by \eqref{eq:bis-density} and \eqref{eq:bis-z-limit}, and letting $M\to+\infty$,
\begin{equation}\label{eq:bis-density-bound}
\frac{d\mu_B}{d\L^d}(x_0)\geq\frac{\sigma_{d-1}}{2}\,a_\lambda\,|u(x_0)|^2\qquad\text{for $\L^d$-almost every }x_0\in B.
\end{equation}
Integrating \eqref{eq:bis-density-bound} over $B$ and recalling \eqref{eq:bis-density-integral}, we obtain
\begin{equation*}
\liminf_{\e\to0}F_\e(u_\e)\geq\frac{\sigma_{d-1}}{2}\,a_\lambda\int_B|u(x)|^2\,dx.
\end{equation*}
Since $B$ is an arbitrary ball, taking an increasing sequence of balls invading $\Rd$, by monotone convergence we conclude that
\begin{equation*}
\liminf_{\e\to0}F_\e(u_\e)\geq\frac{\sigma_{d-1}}{2}\,a_\lambda\|u\|^2_{L^2(\Rd)}=F_\lambda(u),
\end{equation*}
which is the $\Gamma$-liminf inequality.

\section{Limsup inequality}\label{sec:limsup}

The limit functional $F_\lambda$ is continuous with respect to the strong topology of $L^2(\Rd)$; therefore, by a diagonal argument it suffices to prove the existence of a recovery sequence, for every $u$ in a strongly dense subset of $L^2(\Rd)$. In our case, we consider the functions $u\in L^2(\Rd)$ whose support is essentially contained in a cube and which are piecewise constant on dyadic cubes of arbitrarily small side length. Without loss of generality, we may assume that $u$ is supported in the reference cube $[0,1]^d$ and that there exists $n\in\N$ such that
\begin{equation}\label{eq:u-dyadic}
u=\sum_{k\in K_n}u_k\rchi_{Q_k},
\end{equation}
where $u_k\in\R$, $Q_k\coloneqq2^{-n}\big(k+[0,1)^d\big)$ and $K_n\coloneqq\{k\in\Z^d:Q_k\sub[0,1]^d\}=\{0,\dots,2^n-1\}^d$.

We fix $\delta>0$, and let $\varphi\in L^2(Q)$ with zero mean such that
\begin{equation}\label{eq:almost-min-0}
A_\lambda(1,\phi)\leq a_\lambda+\delta.
\end{equation}
The function $\varphi$ is extended to $\Rd$ by $1$-periodicity.
Note that, by a density argument, we may assume $\phi\in \mathcal{C}^\infty(\Rd)$ and that if $\lambda=+\infty$, then $\phi=0$ (see Remark \ref{rem:a-infty}). We then set $L\coloneqq\operatorname{Lip}(\phi)\max_{k\in K_n}|u_k|\,<+\infty$.

We define the recovery sequence as
\begin{equation}\label{eq:recovery}
u_\e(x)\coloneqq u(x)\Big(1+\phi\Big(\frac{x}{\e}\Big)\Big),\qquad x\in\Rd.
\end{equation}
First of all, note that $u_\e\overset{L^2}{\rightharpoonup}u$, and there exists $C>0$, independent of $\e$, such that $\|u_\e\|_{L^\infty(\Rd)}\leq C$. Moreover, for every $k\in K_n$, the restriction of $u_\e-u$ to $Q_k$ is smooth and
\begin{equation}\label{eq:grad-recovery}
\|\nabla(u_\e-u)\|_{L^\infty(Q_k)}= \frac{|u_k|}\e\|\nabla \varphi\|_{L^\infty(\Rd)}\leq\frac{L}{\e}.
\end{equation}
We also point out that $u_\e\in H^t(\Rd)$ for every $t\in(0,1/2)$.

We fix a parameter $R\coloneqq R(\e)>0$ such that
\begin{equation*}
\lim_{\e\to0}R=+\infty,\qquad\lim_{\e\to0}R^{-2s}=1.
\end{equation*}
We also set
\begin{equation*}
\Sigma\coloneqq\bigcup_{k\in K_n}\partial Q_k.
\end{equation*}
Note that $\Sigma$ depends only on $u$, and that it satisfies
\begin{equation}\label{eq:minkowski-Sigma}
\big|\{x\in\Rd:\operatorname{dist}(x,\Sigma)<\tau\}\big|\leq C_\Sigma\tau\qquad\text{for every }\tau\in(0,1],
\end{equation}
where $C_\Sigma>0$ depends only on $d$ and $u$.
We decompose the interactions into several regions, and discuss separately the contributions. 
Given a measurable set $\Lambda\sub\Rd\times\Rd$, we consider the localized functional
\begin{equation}\label{eq:F-localized}
F_\e(v,\Lambda)\coloneqq s\iint_{\Lambda}a\Big(\frac{x}{\e},\frac{y}{\e}\Big)\frac{|v(x)-v(y)|^2}{|x-y|^{d+2s}}\,dx\,dy.
\end{equation}
For $\e$ small enough so that $[0,1]^d\sub B_R$, we divide $\Rd\times\Rd$ into the following sets
\begin{align*}
\Lambda_{1,\e}&\coloneqq\big\{(x,y)\in\Rd\times\Rd:x\notin B_R\text{ or }y\notin B_R\big\},\\
\Lambda_{2,\e}&\coloneqq\bigcup_{k\in K_n}\big\{(x,y)\in Q_k\times Q_k:|x-y|<1\big\},\\
\Lambda_{3,\e}&\coloneqq\big\{(x,y)\in B_R\times B_R:|x-y|\geq 1\big\},\\
\Lambda_{4,\e}&\coloneqq\big\{(x,y)\in B_R\times B_R:|x-y|<1\big\}\setminus\Lambda_{2,\e}.
\end{align*}
We note that $\Lambda_{2,\e}$ does not depend on $\e$, but we include the index to simplify the notation throughout. For every $\e>0$ small enough and every $v\in L^2(\Rd)$, it holds
\begin{equation}\label{eq:decomposition}
F_\e(v)=\sum_{k=1}^4F_\e(v,\Lambda_{k,\e}).
\end{equation}

We show that the contributions coming from the sets $\Lambda_{3,\e}$ and $\Lambda_{4,\e}$ are asymptotically negligible.

We first consider $\Lambda_{3,\e}$. Since $u_\e$ vanishes outside $[0,1]^d$, only the pairs $(x,y)\in\Lambda_{3,\e}$ with $x\in[0,1]^d$ or $y\in[0,1]^d$ contribute, and for such pairs we have $1\leq|x-y|<2R$. Hence, using the symmetry of $a$, $a\leq\beta$ and \cref{lem:annuli} with $A=[0,1]^d$, $r_1=1$ and $r_2=2R$, we obtain
\begin{align*}
F_\e(u_\e,\Lambda_{3,\e})&\leq2\beta s\int_{[0,1]^d}\int_{\{1\leq|x-y|<2R\}}\frac{|u_\e(x)-u_\e(y)|^2}{|x-y|^{d+2s}}\,dy\,dx\\
&\leq4\beta\sigma_{d-1}\sup_{\e}\|u_\e\|^2_{L^\infty(\Rd)}\big(1-(2R)^{-2s}\big).
\end{align*}
Since $R^{-2s}\to1$ and $(u_\e)$ is bounded in $L^\infty(\Rd)$, it follows that
\begin{equation}\label{eq:Lambda3}
\limsup_{\e\to0}F_\e(u_\e,\Lambda_{3,\e})=0.
\end{equation}
We now consider $\Lambda_{4,\e}$. Again, only the pairs $(x,y)\in\Lambda_{4,\e}$ with $x\in[0,1]^d$ or $y\in[0,1]^d$ contribute. For such pairs, $x$ and $y$ do not belong to the same cube $Q_k$, so that the segment joining them intersects $\Sigma$ and $\operatorname{dist}(x,\Sigma)\leq|x-y|<1$. Therefore, by \eqref{eq:minkowski-Sigma} and \cref{lem:jump}, for $\e$ small enough so that $s<1/2$, we get
\begin{align*}
F_\e(u_\e,\Lambda_{4,\e})&\leq\beta s\iint_{\{\operatorname{dist}(x,\Sigma)\leq|x-y|<1\}}\frac{|u_\e(x)-u_\e(y)|^2}{|x-y|^{d+2s}}\,dy\,dx\\
&\leq4\beta\sigma_{d-1}C_\Sigma\sup_{\e}\|u_\e\|^2_{L^\infty(\Rd)}\,\frac{s}{1-2s}.
\end{align*}
Since $s\to0$ and $(u_\e)$ is bounded in $L^\infty(\Rd)$, it follows that
\begin{equation}\label{eq:Lambda4}
\limsup_{\e\to0}F_\e(u_\e,\Lambda_{4,\e})=0.
\end{equation}

We now consider $\Lambda_{1,\e}$. Since $u_\e$ vanishes outside $[0,1]^d$ and $[0,1]^d\sub B_R$, by the symmetry of $a$ we have
\begin{equation*}
F_\e(u_\e,\Lambda_{1,\e})=2\sum_{k\in K_n}s |u_k|^2\int_{Q_k}\int_{\Rd\setminus B_R}a\Big(\frac{x}{\e},\frac{y}{\e}\Big)\frac{|1+\phi(x/\e)|^2}{|x-y|^{d+2s}}\,dy\,dx.
\end{equation*}
For $\e$ small enough, we have $R\geq\max\{2\sqrt d,\sqrt d\,\e\}$. Hence, applying \cref{lem:long-range} with $A=Q_k\sub B_{\sqrt d}$, $\rho=\sqrt d$ and $u=u_\e$, we get
\begin{equation*}
\Big|F_\e(u_\e,\Lambda_{1,\e})-\sigma_{d-1}R^{-2s}\sum_{k\in K_n}|u_k|^2\int_{Q_k}\overline{a}^{y}\Big(\frac{x}{\e}\Big)\Big|1+\phi\Big(\frac{x}{\e}\Big)\Big|^2\,dx\Big|\leq2C_d\beta\,s\,\frac{\sqrt d+\e}{R}R^{-2s}\|u_\e\|^2_{L^2(\Rd)},
\end{equation*}
where $C_d$ is the constant given by \cref{lem:long-range}. Since $s\to0$, $R\to+\infty$, $R^{-2s}\to1$ and $(u_\e)$ is bounded in $L^2(\Rd)$, it follows that
\begin{equation*}
\lim_{\e\to0}\Big(F_\e(u_\e,\Lambda_{1,\e})-\sigma_{d-1}\sum_{k\in K_n}|u_k|^2\int_{Q_k}\overline{a}^{y}\Big(\frac{x}{\e}\Big)\Big|1+\phi\Big(\frac{x}{\e}\Big)\Big|^2\,dx\Big)=0.
\end{equation*}
Finally, since the function $x\mapsto\overline{a}^{y}(x)|1+\phi(x)|^2$ is $1$-periodic and bounded, by the Riemann--Lebesgue lemma, we obtain
\begin{eqnarray}\label{eq:Lambda1}\nonumber
\lim_{\e\to0}F_\e(u_\e,\Lambda_{1,\e})&=&\frac{\sigma_{d-1}}{2}\sum_{k\in K_n}|Q_k||u_k|^2\,2\int_{Q}\overline{a}^{y}(x)|1+\phi(x)|^2\,dx\\
&=&\frac{\sigma_{d-1}}{2} \,\Big(2\int_{Q}\overline{a}^{y}(x)|1+\phi(x)|^2\,dx\Big) \|u\|^2_{L^2(\Rd)}.
\end{eqnarray}

We finally consider $\Lambda_{2,\e}$. We first assume that $\lambda<+\infty$. For $x,y\in Q_k$ we have $u_\e(x)-u_\e(y)=u_k(\phi(x/\e)-\phi(y/\e))$; therefore, enlarging the domain of integration in $y$, we get
\begin{equation*}
F_\e(u_\e,\Lambda_{2,\e})\leq\sum_{k\in K_n}s |u_k|^2\int_{Q_k}\int_{\{|x-y|<1\}}a\Big(\frac{x}{\e},\frac{y}{\e}\Big)\frac{|\phi(x/\e)-\phi(y/\e)|^2}{|x-y|^{d+2s}}\,dy\,dx.
\end{equation*}
We split the inner integral into the contributions of $\{|x-y|\leq\e\}$ and $\{\e<|x-y|<1\}$.

Since $|u_k||\phi(x/\e)-\phi(y/\e)|\leq L|x-y|/\e$, integrating in polar coordinates we get
\begin{align*}
\sum_{k\in K_n}s|u_k|^2\int_{Q_k}\int_{\{|x-y|\leq\e\}}a\Big(\frac{x}{\e},\frac{y}{\e}\Big)\frac{|\phi(x/\e)-\phi(y/\e)|^2}{|x-y|^{d+2s}}\,dy\,dx&\leq\beta L^2\e^{-2}s\,\sigma_{d-1}\int_0^\e t^{1-2s}\,dt\\
&=\frac{\beta\sigma_{d-1}L^2}{2-2s}\,s\,\e^{-2s},
\end{align*}
which vanishes as $\e\to0$, since $s\to0$ and $\e^{-2s}\to\lambda<+\infty$.

Now, we set $b(x,y)\coloneqq a(x,y)|\phi(x)-\phi(y)|^2$, which is a $1$-periodic Borel function with $0\leq b\leq4\beta\|\phi\|^2_{L^\infty(\Rd)}$ and
\begin{equation*}
\overline b=\int_{Q}\int_{Q}a(x,y)|\phi(x)-\phi(y)|^2\,dx\,dy.
\end{equation*}
Applying \cref{lem:averaging} with $A=Q_k$, $c=1$ and outer radius $1$, together with \cref{obs:averaging-Linfty}, for $\e$ small enough  we get
\begin{multline*}
s\int_{Q_k}\int_{\{\e<|x-y|<1\}}a\Big(\frac{x}{\e},\frac{y}{\e}\Big)\frac{|\phi(x/\e)-\phi(y/\e)|^2}{|x-y|^{d+2s}}\,dy\,dx
\leq\frac{\sigma_{d-1}}{2}\big(\e^{-2s}-1\big)\,\overline b\,|Q_k|\\+4C_d\beta\|\phi\|^2_{L^\infty(\Rd)}\,\e^{-2s}\big(2^{d+4}s|Q_k|+2^d|\partial_{2\e}Q_k|\big),
\end{multline*}
where $C_d$ is the constant given by \cref{lem:averaging}. Since $s\to0$, $\e^{-2s}\to\lambda$, and $|\partial_{2\e}Q_k|\to0$, it follows that
\begin{eqnarray}\label{eq:Lambda2}\nonumber
\limsup_{\e\to0}F_\e(u_\e,\Lambda_{2,\e})&\leq&\frac{\sigma_{d-1}}{2}\sum_{k\in K_n}|u_k|^2|Q_k|\,(\lambda-1)\int_{Q}\int_{Q}a(x,y)|\phi(x)-\phi(y)|^2\,dx\,dy\\
&=&\frac{\sigma_{d-1}}{2}(\lambda-1)\int_{Q}\int_{Q}a(x,y)|\phi(x)-\phi(y)|^2\,dx\,dy\, \|u\|^2_{L^2(\Rd)}.
\end{eqnarray}
If $\lambda=+\infty$, then $\phi=0$, so that $F_\e(u_\e,\Lambda_{2,\e})=0$ and \eqref{eq:Lambda2} holds with the convention $(+\infty)\cdot0=0$.

Gathering \eqref{eq:decomposition}, \eqref{eq:Lambda3}, \eqref{eq:Lambda4}, \eqref{eq:Lambda1} and \eqref{eq:Lambda2}, and recalling \eqref{eq:a-lambda-z}, we obtain
\begin{align*}
\limsup_{\e\to0}F_\e(u_\e)&\leq\limsup_{\e\to0}\big(F_\e(u_\e,\Lambda_{1,\e})+F_\e(u_\e,\Lambda_{2,\e})\big)+\sum_{k=3}^4\limsup_{\e\to0}F_\e(u_\e,\Lambda_{k,\e})\\
&\leq\frac{\sigma_{d-1}}{2}A_\lambda(1,\varphi)\|u\|^2_{L^2(\Rd)}.
\end{align*}
By \eqref{eq:almost-min-0}, we conclude that
\begin{equation}\label{eq:limsup-delta}
\limsup_{\e\to0}F_\e(u_\e)\leq F_\lambda(u)+\frac{\sigma_{d-1}}{2}\,\delta\|u\|^2_{L^2(\Rd)}.
\end{equation}
By the arbitrariness of $\delta\to0$ and a diagonal argument, this gives the the $\Gamma$-limsup inequality.

\section{The case of a general $1<p<\infty$}
\label{sec:pneq2}
In this section, we briefly comment on the modifications needed to prove the analogue of Theorem \ref{thm:main} in the case of $1<p<\infty$. The preparatory results of Section \ref{subsec:asymptotic} hold for general $p$, {\em mutatis mutandis}, and similarly for the strategies carried out in the $\Gamma$-liminf/limsup inequalities. The only role of $p=2$ is played in the $\Gamma$-liminf inequality, for expanding $|u(x)-u(y)|^2$ and consequently proving an estimate of the form
\begin{equation}\label{eq:intro-p}
2s\int_B\int_{\R^d\setminus B_R}a\Big(\frac{x}{\e},\frac{y}{\e}\Big)\frac{|u_\e(x)-u_\e(y)|^2}{|y|^{d+2s}}\,dy\,dx\geq2s\int_B\int_{\R^d\setminus B_R}a\Big(\frac{x}{\e},\frac{y}{\e}\Big)\frac{|u_\e(x)|^2}{|y|^{d+2s}}\,dy\,dx+o(1),
\end{equation}
where $B\subset\Rd$ is a fixed ball and $R\gg1$. Indeed, expanding the square and throwing away the contribution from $|u(y)|^2$, one obtains an error where $u(y)$ appears to be raised to a lower power. This allows us to bound it by using that the $L^2$-norm of the singular kernels $s|y|^{-d-2s}$ in the region $\Rd\setminus B_R$ goes to zero as $R\to+\infty$ and $s\to0$. 

In the case of a general $p$, there is a natural replacement of this procedure. Indeed, for every $t,\tau\in\R$ it holds
\begin{equation*}
\big||t-\tau|^p-|t|^p\big|\leq C_p\big(|t|^{p-1}|\tau|+|\tau|^p\big),\qquad C_p\coloneqq p\max\{1,2^{p-2}\},
\end{equation*}
so that, integrating $\big||u_\e(x)-u_\e(y)|^p-|u_\e(x)|^p\big|$ against $s|y|^{-d-sp}$, the error in the analog of \eqref{eq:intro-p} is bounded by
\begin{equation*}
  2\beta C_p\,s\int_B\int_{\Rd\setminus B_R}\frac{|u_\e(x)|^{p-1}|u_\e(y)|}{|y|^{d+sp}}\,dy\,dx
+2\beta C_p\,s\int_B\int_{\Rd\setminus B_R}\frac{|u_\e(y)|^{p}}{|y|^{d+sp}}\,dy\,dx.  
\end{equation*}
In the formula above, the second term is bounded by  $C_{B,d,p} s\,R^{-d-sp}\sup_{\e>0}\|u_\e\|^p_{L^p(\Rd)}$, which is infinitesimal, while it can be shown that the first term goes to zero by improving the integrability of the kernels as in the case $p=2$, using a bound in some $L^{1+\delta}$ with $\delta>0$ instead of $L^2$.

Finally, the complete statement in the general case is as follows.
\begin{theo}[Homogenized Maz'ya--Shaposhnikova]\label{thm:main2}
Let $a\colon\Rd\times\Rd\to\R$ be a  Borel $1$-periodic function such that there exist $0<\alpha\leq\beta$ for which $\alpha\leq a(x,y)\leq\beta$ for all $x,y\in\Rd$. Let $1<p<\infty$, $\e>0$ and $s\coloneqq s(\e)\in(0,1)$. Let
\begin{equation}
F_{\e,p}(u)\coloneqq s\int_{\Rd}\int_{\Rd}a\Big(\frac{x}{\e},\frac{y}{\e}\Big)\frac{|u(x)-u(y)|^p}{|x-y|^{d+sp}}\,dx\,dy
\end{equation}
be defined in $L^p(\Rd)$. If $\lim_{\e\to0}s=0$ and
\begin{equation*}
\lim_{\e\to0}\e^{-sp}=\lambda\in[1,+\infty],
\end{equation*}
then
\begin{equation*}
\Gamma\text{-}\lim_{\e\to0}F_{\e,p}(u)=\frac{\sigma_{d-1}}{p}\,a_{\lambda,p}\|u\|^p_{L^p(\Rd)},
\end{equation*}
for all $u\in L^p(\Rd)$, where the  sequential $\Gamma$-limit is computed with respect to the weak topology of $L^p(\Rd)$, and $a_{\lambda,p}$ is given by 
\begin{eqnarray}\nonumber
&&\hskip-1.3cm a_{\lambda,p}\coloneqq\inf\bigg\{\int_{Q}\int_{Q}(a(x,y)+a(y,x))|1+\phi(x)|^p\,dx\,dy\\&&\ \quad+(\lambda-1)\int_{Q}\int_{Q} a(x,y)|\phi(x)-\phi(y)|^p\,dx\,dy:\phi\in L^p(Q)\text{ 
and }\int_Q\phi(x)\,dx=0\bigg\}
\end{eqnarray}
if $\lambda<+\infty$, while
\begin{eqnarray}
a_{+\infty,p}\coloneqq\ 2\int_{Q}\int_{Q}a(x,y)dx\,dy.
\end{eqnarray}
\end{theo}

\section*{Acknowledgments}
Andrea Braides acknowledges funding by the Excellence Project MatMod@TOV (Grant Agreement No.~CUP\_E83C23000330006) of the
Department of Mathematics of the University of Rome Tor Vergata. 
Marco Picerni acknowledges funding by the Italian Ministry of University and Research (MUR) under the `Fondo Italiano per la Scienza' (FIS 3) program, PI: prof.\ Nicola Gigli, Project Title: Modern challEnges in Geometric Analysis - MEGA, CUP: G53C25000920001. All authors are members of the GNAMPA group of INdAM.

\subsection*{Statements}
The authors declare no conflict of interest. The work has no associated data.
For open access purposes, the authors have applied
a CC BY public copyright license to any accepted manuscript version arising from this
submission.

\printbibliography

@article {FM,
    AUTHOR = {Fonseca, Irene and M\"uller, Stefan},
     TITLE = {Quasi-convex integrands and lower semicontinuity in {$L^1$}},
   JOURNAL = {SIAM J. Math. Anal.},
  FJOURNAL = {SIAM Journal on Mathematical Analysis},
    VOLUME = {23},
      YEAR = {1992},
    NUMBER = {5},
     PAGES = {1081--1098},
      ISSN = {0036-1410},
   MRCLASS = {49J45 (49Q20)},
  MRNUMBER = {1177778},
MRREVIEWER = {Anna\ Salvadori},
       DOI = {10.1137/0523060},
       URL = {https://doi.org/10.1137/0523060},
}

@article {BMS,
    AUTHOR = {Braides, Andrea and Maslennikov, Mikhail and Sigalotti, Laura},
     TITLE = {Homogenization by blow-up},
   JOURNAL = {Appl. Anal.},
  FJOURNAL = {Applicable Analysis. An International Journal},
    VOLUME = {87},
      YEAR = {2008},
    NUMBER = {12},
     PAGES = {1341--1356},
      ISSN = {0003-6811,1563-504X},
   MRCLASS = {35B27 (49J45)},
  MRNUMBER = {2482386},
MRREVIEWER = {Dmitry\ Golovaty},
       DOI = {10.1080/00036810802555458},
       URL = {https://doi.org/10.1080/00036810802555458},
}

@article {BBD,
    AUTHOR = {Braides, Andrea and Brusca, Giuseppe Cosma and Donati, Davide},
     TITLE = {Another look at elliptic homogenization},
   JOURNAL = {Milan J. Math.},
  FJOURNAL = {Milan Journal of Mathematics},
    VOLUME = {92},
      YEAR = {2024},
    NUMBER = {1},
     PAGES = {1--23},
      ISSN = {1424-9286,1424-9294},
   MRCLASS = {49J45 (35B27 35R11)},
  MRNUMBER = {4747599},
MRREVIEWER = {Luca\ Lussardi},
       DOI = {10.1007/s00032-023-00389-y},
       URL = {https://doi.org/10.1007/s00032-023-00389-y},
}

@article{Quartet,
title = {Asymptotics of the $s$-perimeter as $s\searrow 0$},
journal = {Discrete and Continuous Dynamical Systems},
volume = {33},
number = {7},
pages = {2777-2790},
year = {2013},
issn = {1078-0947},
doi = {10.3934/dcds.2013.33.2777},
author = {Serena Dipierro and Alessio Figalli and Giampiero Palatucci and Enrico Valdinoci},
}

@misc{CaselliGennaioli,
      title={Asymptotics as $s \to 0^+$ of the fractional perimeter on Riemannian manifolds}, 
      author={Michele Caselli and Luca Gennaioli},
      year={2024},
      eprint={2306.11590},
      archivePrefix={arXiv},
      primaryClass={math.DG},
      url={https://arxiv.org/abs/2306.11590}, 
}

@misc{DaDfGioPi-Necess-Suff,
      title={Necessary and Sufficient Conditions for the Maz'ya-Shaposhnikova Formula in (Fractional) Sobolev Spaces}, 
      author={Elisa Davoli and Giovanni Di Fratta and Rossella Giorgio and Andrea Pinamonti},
      year={2025},
      eprint={2509.23226},
      archivePrefix={arXiv},
      primaryClass={math.AP},
      url={https://arxiv.org/abs/2509.23226}, 
}

@article{MazyaShaposhnikova2002,
  author  = {Maz'ya, V. and Shaposhnikova, T.},
  title   = {On the Bourgain, Brezis, and Mironescu theorem concerning limiting embeddings of fractional Sobolev spaces},
  journal = {Journal of Functional Analysis},
  volume  = {195},
  year    = {2002},
  pages   = {230--238}
}

@misc{HanPinamontirigidity,
      title={On the asymptotic behaviour of the fractional Sobolev seminorms in metric measure spaces: asymptotic volume ratio, volume entropy and rigidity}, 
      author={Han, Bang-Xian and Pinamonti, A.},
      year={2024},
      eprint={2108.06996},
      archivePrefix={arXiv},
      primaryClass={math.FA},
      url={https://arxiv.org/abs/2108.06996}, 
}

@article{GaussianFracPer,
    author = {Carbotti, A. and Cito, S. and La Manna, D.A. and Pallara, D.},
    title = {Asymptotics of the s-fractional Gaussian perimeter as $s\rightarrow 0^+$},
    journal = {Fract. Calc. Appl. Anal.},
    year = {2022},
    volume = {25},
    pages={1388–1403},
    doi={https://doi.org/10.1007/s13540-022-00066-8}
}

@article{HanPinamontiXu2025,
  author    = {Han, Bang-Xian and Pinamonti, A. and Xu, Z. and Zambanini, K.},
  title     = {Maz’ya--Shaposhnikova Meet Bishop--Gromov},
  journal   = {Potential Analysis},
  volume    = {63},
  pages     = {513--529},
  year      = {2025},
  doi       = {10.1007/s11118-024-10179-9},
  url       = {https://doi.org/10.1007/s11118-024-10179-9}
}

@article{Han-MS-BBM,
title = {On the asymptotic behaviour of the fractional Sobolev seminorms: A geometric approach},
journal = {Journal of Functional Analysis},
volume = {287},
number = {9},
pages = {110608},
year = {2024},
issn = {0022-1236},
doi = {https://doi.org/10.1016/j.jfa.2024.110608},
author = {Bang-Xian Han}
}

@article {CriDLKNP23,
    AUTHOR = {Crismale, V. and De Luca, L. and Kubin, A. and Ninno, A. and
              Ponsiglione, M.},
     TITLE = {The variational approach to {$s$}-fractional heat flows and
              the limit cases {$s\to0^+$} and {$s\to1^-$}},
   JOURNAL = {J. Funct. Anal.},
  FJOURNAL = {Journal of Functional Analysis},
    VOLUME = {284},
      YEAR = {2023},
    NUMBER = {8},
     PAGES = {Paper No. 109851, 38},
      ISSN = {0022-1236,1096-0783},
   MRCLASS = {49J45 (35K20 35R11)},
  MRNUMBER = {4544090},
MRREVIEWER = {Leszek\ Gasi\'{n}ski},
       DOI = {10.1016/j.jfa.2023.109851},
       URL = {https://doi.org/10.1016/j.jfa.2023.109851},
}

@article{bourgain:hal-00747692,
  TITLE = {{Another look at Sobolev spaces}},
  AUTHOR = {Bourgain, Jean and Brezis, Ha{\"i}m and Mironescu, Petru},
  URL = {https://hal.science/hal-00747692},
  NOTE = {Original research article appeared at in Optimal Control and Partial Differential Equations IOS Press ISBN 1 58603 096 5},
  PAGES = {439-455},
  YEAR = {2001},
  HAL_ID = {hal-00747692},
  HAL_VERSION = {v1},
}

@misc{BMS26,
      title={Homogenization and averaging in the {B}ourgain--{B}rezis--{M}ironescu limit}, 
      author={Andrea Braides and Stefano Mannella and Alec Jacopo Almo Schiavoni~Piazza},
      year={2026},
      eprint={2609.21933},
      archivePrefix={arXiv},
      primaryClass={math.AP},
      url={https://arxiv.org/abs/2609.21933}, 
}

@misc{PiniSantilli2026,
      title={$\Gamma$-convergence, variational analysis and characterisation of minimisers for $(s,p)$-Gagliardo energies in the flat $d$-torus}, 
      author={G. Pini and F. Santilli},
      year={2026},
      eprint={2604.25520},
      archivePrefix={arXiv},
      primaryClass={math.FA},
      url={https://arxiv.org/abs/2604.25520}, 
}

@misc{DaFaPi1,
    author = {Elisa Davoli and Alberto Fanizza and Marco Picerni},
    title = {Closing the gap: {Maz'ya--Shaposhnikova} and asymptotics of fractional perimeters},
    year = {2026},
    eprint = {2605.03955},
    archivePrefix = {arXiv},
    primaryClass = {math.AP},
    doi = {10.48550/arXiv.2605.03955},
    url = {https://arxiv.org/abs/2605.03955}
}
\ \bigskip

\end{document}